\documentclass[final,leqno,letterpaper]{etna}
\usepackage[english]{babel}
\usepackage[utf8]{inputenc}
\usepackage{amsmath}
\usepackage{amssymb,enumerate,algorithm}
\usepackage{graphicx}

\usepackage{lineno}
\usepackage[style=numeric,sortcites=true,mincrossrefs=9999,date=year,giveninits]{biblatex}
\usepackage{xurl}

\renewcommand\bibfont\footnotesize

\makeatletter
\defbibheading{bibliography}[\bibname]{%
\setlength{\parskip}{0mm}
\par\addvspace{18pt}
\@beginparpenalty=10000
{\centering \footnotesize\uppercase\expandafter{references}\par \nobreak \addvspace{11pt} \nobreak \par}}
\makeatother

\DeclareFieldFormat*{title}{\mkbibemph{#1}}
\DeclareFieldFormat*{citetitle}{\mkbibemph{#1}}
\DeclareFieldFormat{journaltitle}{#1}

\renewbibmacro*{in:}{%
  \ifentrytype{article}
    {}
    {\printtext{\bibstring{in}\intitlepunct}}}

\newbibmacro*{pubinstorg+location+date}[1]{%
  \printlist{#1}%
  \newunit
  \printlist{location}%
  \newunit
  \usebibmacro{date}%
  \newunit}

\renewbibmacro*{publisher+location+date}{\usebibmacro{pubinstorg+location+date}{publisher}}
\renewbibmacro*{institution+location+date}{\usebibmacro{pubinstorg+location+date}{institution}}
\renewbibmacro*{organization+location+date}{\usebibmacro{pubinstorg+location+date}{organization}}
\AtEveryBibitem{
  \clearlist{language}
  \clearfield{isbn}
  \clearfield{issn}
  \clearfield{urlyear}
  \clearfield{urlmonth}
}

\newcommand{\Omicron}{\mathrm{O}}
\newcommand{\assign}{:=}
\newcommand{\backassign}{=:}
\newcommand{\mathd}{\mathrm{d}}
\newcommand{\mathi}{\mathrm{i}}
\newcommand{\nosymbol}{}
\newcommand{\of}{:}
\newcommand{\tmem}[1]{{\em #1\/}}
\newcommand{\tmmathbf}[1]{\ensuremath{\boldsymbol{#1}}}
\newcommand{\tmop}[1]{\ensuremath{\operatorname{#1}}}
\newcommand{\tmscript}[1]{\text{\scriptsize{$#1$}}}

\newenvironment{enumeratenumeric}{\begin{enumerate}[1.] }{\end{enumerate}}
\newcommand{\minmaj}[1]{{}^{\sharp} #1}
\newcommand{\abs}[1]{\mathopen|#1\mathclose|}

\hypersetup{%
  pdftitle={Rounding Error Analysis of Linear Recurrences
            using Generating Series},
  pdfauthor={Marc Mezzarobba},
  pdfkeywords={
    rounding error,
    rigorous computing,
    complex variable,
    majorant series,
    Bernoulli numbers,
    vibrating string,
    differentially finite function%
  }
}

\title{Rounding Error Analysis of Linear Recurrences\\
using Generating Series}
\author{Marc Mezzarobba\footnotemark[2]}

\shorttitle{Error analysis of recurrences using generating series}
\shortauthor{M. Mezzarobba}

\begin{document}

\maketitle

\begingroup
\renewcommand{\thefootnote}{\fnsymbol{footnote}}
\footnotetext[2]{%
  Sorbonne Université, CNRS, LIP6, F-75005 Paris, France;
  \emph{current address:} LIX, CNRS, École polytechnique, Institut polytechnique de Paris, 91120 Palaiseau, France
  (\texttt{\href{mailto:marc@mezzarobba.net}{marc@mezzarobba.net}})
  % web: \url{http://marc.mezzarobba.net/}%
}
\endgroup

\begingroup
\renewcommand{\thefootnote}{\fnsymbol{footnote}}
\footnotetext{%
This work is licensed under the Creative Commons Attribution 4.0 International License. To view a copy of this license, visit
\url{http://creativecommons.org/licenses/by/4.0/}
or send a letter to Creative Commons, PO Box 1866, Mountain View, CA 94042, USA.}
\endgroup

\begin{abstract}
  We develop a toolbox for the error analysis of linear recurrences with
  constant or polynomial coefficients, based on generating series, Cauchy's
  method of majorants, and simple results from analytic combinatorics. We
  illustrate the power of the approach by several nontrivial application
  examples. Among these examples are a new worst-case analysis of an algorithm
  for computing Bernoulli numbers, and a new algorithm for evaluating
  differentially finite functions in interval arithmetic while avoiding
  interval blow-up.
\end{abstract}



\section{Introduction}

This article aims to illustrate a technique for bounding round-off errors in
the floating-point evaluation of linear recurrence sequences that we found to
work well on a number of interesting examples. The main idea is to encode as
generating series both the sequence of ``local'' errors committed at each step
and that of ``global'' errors resulting from the accumulation of local errors.
While the resulting bounds are unlikely to be surprising to specialists,
generating series techniques, curiously, do not seem to be classical in this
context.

As is well-known, in the evaluation of a linear recurrence sequence, rounding
errors typically cancel out to a large extent instead of purely adding up. It
is crucial to take this phenomenon into account in the analysis in order to
obtain realistic (worst-case) bounds, which makes it necessary to study the propagation of
local errors in the following steps of the algorithm somewhat finely. In the
classical language of sequences, this tends to involve complicated
manipulations of nested sums and yield opaque expressions.

Generating series prove a convenient alternative for several reasons. Firstly,
they lead to more manageable formulae: convolutions become products, and the
relation between the local and the accumulated errors can often be expressed
exactly as an algebraic or differential equation involving their generating
series. Secondly, such an equation opens the door to powerful analytic
techniques like singularity analysis or Cauchy's method of majorants. Thirdly,
as illustrated in Section~\ref{sec:bernoulli}, a significant part of the
laborious calculations involved in obtaining explicit constants can be carried
out with the help of computer algebra systems when the calculation is
expressed using series.

In this article, we substantiate our claim that generating series are an
adequate language for error analysis by applying it to a selection of examples
from the literature.
Our focus is on true mathematical bounds (as opposed, in particular, to linearized bounds) on worst-case errors (with no assumptions on the distribution of individual rounding errors).
As detailed below, some of the examples yield results
that appear to be new and may be of independent interest.

The text is organized as follows. In order to get a concrete feeling of the
basic idea, we start in Section~\ref{sec:toy abs} with an elementary example,
postponing to Section~\ref{sec:related} the discussion of related work. We
continue in Sections \ref{sec:genseries}~to~\ref{sec:fp} with a review of
classical facts about generating series, asymptotics, the Cauchy majorant
method, and floating-point error analysis that constitute our basic toolbox.
Building on this background, we then illustrate the approach outlined in
Section~\ref{sec:toy abs} in situations involving polynomial coefficients
(Legendre polynomials, Section~\ref{sec:legendre}) and floating-point
arithmetic (with a revisit of the toy example in Section~\ref{sec:toy rel}). A
reader only interested in understanding the method can stop there.

The second half of the article presents more substantial applications of the same idea.
It consists of three sections that can be read independently, referring to Sections
\ref{sec:genseries}~to~\ref{sec:fp} for basic results as necessary.
Section~\ref{sec:bernoulli} answers a question of R.~P.~Brent and
P.~Zimmermann on the floating-point computation of Bernoulli numbers.
Section~\ref{sec:wave} discusses variations on Boldo's~{\cite{Boldo2009}}
error analysis of a finite difference scheme for the 1D wave equation, using
series with coefficients in a normed algebra to encode a bivariate recurrence.
Finally, in Section~\ref{sec:dfinite}, we take a slightly different
perspective and ask how to evaluate the sum of a series whose coefficients
satisfy a recurrence when the recurrence is part of the input. Under mild
assumptions, we give an algorithm that computes a rigorous enclosure of the
sum while avoiding the exponential blow-up that would come with a naive use of
interval arithmetic.

\section{A Toy Example}\label{sec:toy abs}

Our first example is borrowed from Boldo~{\cite[Section~2.1]{Boldo2009}} and
deals with the evaluation of a very simple, explicit linear recurrence
sequence with constant coefficients in a simple model of approximate
arithmetic. It is not hard to carry out the error analysis in classical
sequence notation, cf.~{\cite{Boldo2009}}, and the reader is encouraged to
duplicate the reasoning in his or her own favorite language.

Consider the sequence $(c_n)_{n \geqslant - 1}$ defined by the recurrence
\begin{equation}
  c_{n + 1} = 2 c_n - c_{n - 1} \label{eq:toy exact}
\end{equation}
with $c_{- 1} = 0$ and a certain initial value~$c_0$.
The exact solution is $c_n = (n + 1) c_0$.
Let us assume that we
are computing this sequence iteratively, so that each iteration generates a
small {\tmem{local error}} corresponding to the evaluation of the right-hand
side. We denote by $(\tilde{c}_n)$ the sequence of computed values. The local
errors accumulate over the course of the computation, and our goal is to bound
the {\tmem{global error}} $\delta_n = \tilde{c}_n - c_n$.

We assume that each arithmetic operation produces an error bounded by a fixed
quantity~$u$.
(This model is similar to fixed-point arithmetic, but our example is simplified to the point of being completely unrealistic:
in an actual fixed-point implementation,
since the coefficients on the right-hand side of~\eqref{eq:toy exact} are integers, this formula involves no rounding error at all.)
Thus, we have
\begin{equation}
  \tilde{c}_{n + 1} = 2 \tilde{c}_n - \tilde{c}_{n - 1} + \varepsilon_n,
  \qquad \abs{\varepsilon_n} \leqslant 2 u \label{eq:toy fxp}
\end{equation}
for all $n \geqslant 0$. We will also assume $\abs{\delta_0} \leqslant u$.
Subtracting \eqref{eq:toy exact} from \eqref{eq:toy fxp} yields
\begin{equation}
  \delta_{n + 1} = 2 \delta_n - \delta_{n - 1} + \varepsilon_n, \quad n
  \geqslant 0, \label{eq:toy rec err}
\end{equation}
with $\delta_{- 1} = 0$.

A naive forward error analysis would have us write $\abs{\delta_{n + 1}}
\leqslant 2 \abs{\delta_n} + \abs{\delta_{n - 1}} + 2 u$ and conclude by induction
that $\abs{\delta_n} \leqslant 3^{n + 1} u$, or, with a bit more effort,
\begin{equation}
  \abs{\delta_n} \leqslant (\lambda_+ \alpha_+^n + \lambda_- \alpha_-^n - 4) u,
  \qquad \alpha_{\pm} = 1 \pm \sqrt{2}, \quad \lambda_{\pm} = 4 \pm 3
  \sqrt{2} . \label{eq:toy fwd}
\end{equation}
Neither of these bounds is satisfactory. To see why, it may help to consider
the propagation of the first few rounding errors. Writing $\tilde{c}_0 = c_0 +
\delta_0$ and $\tilde{c}_1 = 2 \tilde{c}_0 + \varepsilon_0 = c_1 + 2 \delta_0
+ \varepsilon_0$, we have
\[ \tilde{c}_2 = 2 (2 c_0 + 2 \delta_0 + \varepsilon_0) - (c_0 + \delta_0) +
   \varepsilon_1 = c_2 + 3 \delta_0 + 2 \varepsilon_0 + \varepsilon_1 \]
and hence $\abs{\delta_2} \leqslant 9 u$. The naive analysis effectively puts
absolute values in this expression, leading to $\abs{\delta_2} \leqslant 5
\abs{\delta_0} + 2 \abs{\varepsilon_0} + \abs{\varepsilon_1} \leqslant 11 u$ instead.
Overestimations of this kind compound as $n$~increases. Somehow keeping track
of the expression of $\delta_n$ as a linear combination of the~$\varepsilon_i$
(and $\delta_0$) clearly should yield better estimates.

To do so, let us note that~\eqref{eq:toy rec err} is a linear recurrence with
the same homogeneous part as~\eqref{eq:toy exact} and the sequence of local
errors on the right-hand side, and rephrase this relation in terms of
generating series. Define the formal power series\footnote{While the sequence
$(\delta_n)$ naturally starts at $n = - 1$, the fact that~$\delta_{- 1} = 0$
allows us to use the same summation range for both series.}
\[ \delta (z) = \sum_{n \geqslant 0} \delta_n z^n, \qquad \varepsilon (z) =
   \sum_{n \geqslant 0} \varepsilon_n z^n . \]
The relation~\eqref{eq:toy rec err} implies
\[ (z^{- 1} - 2 + z) \delta (z) = \sum_{n \geqslant - 1} \delta_{n + 1} z^n -
   2 \sum_{n \geqslant 0} \delta_n z^n + \sum_{n \geqslant 1} \delta_{n - 1}
   z^n = \delta_0 z^{- 1} + \sum_{n \geqslant 0} \varepsilon_n z^n, \]
that is,
\begin{equation}\label{eq:toy abs delta}
  \delta (z) = \frac{\delta_0 + z \varepsilon (z)}{(1 - z)^2}.
\end{equation}
Since $\abs{\delta_0} \leqslant u$ and $\abs{\varepsilon_n} \leqslant 2 u$, we see
that the absolute values of the coefficients of the numerator are bounded by
those of the corresponding coefficients in the series expansion of $2 u / (1 -
z)$. Denoting by $\ll$ this termwise inequality relation, it follows that
\[ \delta (z) \ll \frac{2 u}{1 - z}  \frac{1}{(1 - z)^2} = \frac{2 u}{(1 -
   z)^3} . \]
Going back to the coefficient sequences, this bound translates into $|
\delta_n | \leqslant ({n + 1}) ({n + 2}) u$, a~much sharper result
than~\eqref{eq:toy fwd}.
This result is essentially optimal in our model, since the~$\varepsilon_n$ might all be equal to~$2u$ and \eqref{eq:toy abs delta} is an exact expression of the global error.

\section{Related Work}\label{sec:related}

There is a large body of literature on numerical aspects of linear recurrence
sequences, especially solutions of three-term recurrences.
The main focus is on on stability issues and backward recurrence
algorithms---algorithms where the recurrence relation is used for decreasing~$n$ and combined with
asymptotic information on the sequence, typically to compute minimal
solutions.
An important early example this nature is Olver's error analysis~\cite{Olver1964} of Miller's method for computing the minimal solution of a second-order recurrence.
We refer to Wimp's book~{\cite{Wimp1984}} for further references.

Here we only consider linear recurrences used in the
forward direction.
Comparatively little has been written on that subject, in Wimp's words, ``not because a forward algorithm is more
difficult to analyze, but rather for the opposite reason---that its analysis
was considered straightforward''~{\cite{Wimp1972}}. The first completely
explicit error analysis of general linear recurrences that we are aware of
appears in the work of Oliver~{\cite[Section~2]{Oliver1967}} (see
also~{\cite{Oliver1965}}).
However, the importance of using linearity to study
the propagation of local errors was recognized well before.
For example,
it is apparent in Clenshaw's discussion~\cite{Clenshaw1955} of his algorithm for computing partial sums of Chebyshev series,
and the
first of Henrici's books on numerical methods for differential
equations~{\cite[Section~1.4]{Henrici1962}} uses the terms ``local round-off
error'' and ``accumulated round-off error'' with the same meaning as we do.

In the same vein as Oliver's work,
Barrio, Melendo, and Serrano~\cite{BarrioMelendoSerrano2003}
analyze the floating-point evaluation of general linear recurrences of finite order.
Their result is a first-order bound, meaning that
terms of order~$\Omicron(u^2)$ where $u$~is the unit roundoff
are omitted.
Furthermore, due to its generality, the bound is complicated and expressed in terms of quantities that may be difficult to estimate.
We believe that the approach presented here offers at least a partial remedy to these limitations.
In more specific situations, though, readily exploitable bounds are available in the literature.
This includes in particular algorithms based on linear recurrences for evaluating finite generalized Fourier series, like Clenshaw's method
\cite[e.g.,][]{Elliott1968,Barrio2002}.

Linear recurrences can also be viewed as special cases of triangular systems of
linear equations. For example, computing the first $n$~terms of the sequence
$(c_n)$~of the previous section is the same as solving the banded Toeplitz
system
\begin{equation}\label{eq:linsys}
  \begin{bmatrix}
     1 &  &  &  & \\
     - 2 & 1 &  & \tmmathbf{0} & \\
     1 & - 2 & 1 &  & \\
     & \ddots & \ddots & \ddots & \\
     \tmmathbf{0} &  & 1 & - 2 & 1
   \end{bmatrix} \begin{bmatrix}
     x_0\\
     x_1\\
     x_2\\
     \vdots\\
     x_n
   \end{bmatrix} = \begin{bmatrix}
     c_0\\
     0\\
     0\\
     \vdots\\
     0
   \end{bmatrix}.
\end{equation}
The study of systems of this type is literally as old as error analysis:
the solution of triangular systems appears as an almost trivial subproblem in von Neumann and
Goldstine's {\cite{vonNeumannGoldstine1947}}\footnote{See also Grcar's
commentary~{\cite[Section 4.5]{Grcar2011a}}.} and (more explicitly) Turing's
{\cite[Section~12, p.~306]{Turing1948}} landmark analyses of linear system
solving, both concluding in a polynomial growth with~$n$ of the forward error
when some quantities related to the inverse or the condition number of the
matrix are fixed. We refer to the encyclopedic book by
Higham~{\cite[Chapter~8]{Higham2002}} for a detailed discussion of the error
analysis of triangular systems and further historical perspective.

Because of their dependency on condition numbers, these results do not, in
themselves, rule out an exponential buildup of errors in the case of
recurrences. In the standard modern proof, the forward error bound results
from the combination of a backward error bound and a perturbation analysis
that could in principle be refined to deal specifically with recurrences. An
issue with this approach is that, to view the numeric solution as the exact
solution corresponding to a perturbed input, one is led to perturb the matrix
in a fashion that destroys the structure inherited from the recurrence.
Experiments by
Barrio, Melendo, and Serrano~\cite{BarrioMelendoSerrano2003}
confirm that their bounds tend to be much sharper than bounds based on the condition number of systems of the type~\eqref{eq:linsys}.

It may nevertheless be the case that one can derive meaningful bounds for
recurrences from a refined variant of Theorem~8.5 in~{\cite{Higham2002}} better
taking into account the structure of the matrix.
Our claim is that
the tools of the present paper are better suited to the task. The use of
linearity to study error propagation can also be viewed as an instance of
backward error analysis, where one chooses to perturb the right-hand side of
the system instead of the matrix. From this perspective, the present paper is
about a convenient way of carrying out the perturbation analysis that enables
one to pass to a forward error bound.

Except for the earlier publication~{\cite{JohanssonMezzarobba2018}} of the
example considered again in Section~\ref{sec:legendre} below,
we are not aware of any prior example of error analysis conducted using
generating series in numerical analysis, scientific computing or computer
arithmetic. A~close analogue appears however in the realm of digital signal
processing, with the use of the Z\mbox{-}transform to study the propagation of
rounding errors in realization of digital filters starting with Liu and
Kaneko~{\cite{LiuKaneko1969}}. The focus in signal processing is rarely on
worst-case error bounds, with the notable exception of recent work by Hilaire
and collaborators {\cite[e.g.,][]{HilaireLopez2013}}.

\section{Generating Series}\label{sec:genseries}

Let $R$ be a ring, typically $R =\mathbb{R}$ or $R =\mathbb{C}$. We denote by
$R [[z]]$ the ring of formal power series
\[ u (z) = \sum_{n = 0}^{\infty} u_n z^n, \]
where $(u_n)_{n = 0}^{\infty}$ is an arbitrary sequence of elements of~$R$. A
series $u (z)$ used primarily as a convenient encoding of its coefficient
sequence~$(u_n)$ is called the {\tmem{generating series}} of~$(u_n)$.

It is often convenient to extend the coefficient sequence to negative indices
by setting $u_n = 0$ for $n < 0$. We can then write $u (z) = \sum_n u_n z^n$
with the implicit summation range extending from $- \infty$ to~$\infty$
(keeping in mind that the product of series of this form does not make sense
in general if the coefficients are allowed to take nonzero values for
arbitrary negative~$n$).

Given $u \in R [[z]]$ and $n \in \mathbb{Z}$, we denote by $u_n$ or $[z^n] u
(z)$ the coefficient of~$z^n$ in $u (z)$. Conversely, whenever $(u_n)$ is a
numeric sequence with integer indices, $u (z)$ is its generating series. We
occasionally consider sequences $u_0 (z), u_1 (z), \ldots$ of series, with
$u_{i, n} = [z^n] u_i (z)$ in this case. We often identify expressions
representing analytic functions with their series expansions at the origin.
For instance, $[z^n]  {(1 - \alpha z)^{- 1}}$ is the coefficient of~$z^n$ in the
Taylor expansion of $(1 - \alpha z)^{- 1}$ at~$0$, that is, $\alpha^n$.

We denote by~$S$ the forward shift operator mapping a sequence $(u_n)_{n \in \mathbb{Z}}$ to $(u_{n + 1})_{n
\in \mathbb{Z}}$, and by $S^{- 1}$ its inverse. Thus, $S \cdot (u_n)_{n \in
\mathbb{Z}}$ is the coefficient sequence of the series~$z^{- 1} u (z)$. More
generally, it is well-known that linear recurrence sequences with constant
coefficients correspond to rational functions in the realm of generating
series, as in the toy example from Section~\ref{sec:toy abs}.

It is also classical that the correspondence generalizes to recurrences with
variable coefficients depending polynomially on~$n$ as follows. We consider
recurrence relations
\begin{equation}\label{eq:rec}
  p_0 (n) u_n + p_1 (n) u_{n - 1} + \cdots + p_s (n) u_{n - s} = b_n, \quad n
  \in \mathbb{Z},
\end{equation}
where $p_0, \ldots, p_s \in R[X]$ are polynomials with $p_0 \neq 0$.
Given sequences expressed in terms
of an index called~$n$, we also denote by~$n$ the operator
\[ (u_n)_{n \in \mathbb{Z}} \quad \mapsto \quad (nu_n)_{n \in \mathbb{Z}} . \]
We then have $Sn = (n + 1) S$, where the product stands for the composition of
operators.

\begin{example}
  With these conventions,
  \[ (nS + n - 1) \cdot (u_n)_{n \in \mathbb{Z}} = (nu_{n + 1} + (n - 1)
     u_n)_{n \in \mathbb{Z}} = ((S + 1)  (n - 1)) \cdot (u_n)_{n \in
     \mathbb{Z}} \]
  is an equality of sequences that parallels the operator equality
  $nS + n - 1 = {(S + 1)}  (n - 1)$.
\end{example}

Any linear recurrence operator of finite order with polynomial coefficients
can thus be written as a polynomial in $n$ and $S^{\pm 1}$. Denoting with a
dot the action of operators on sequences, \eqref{eq:rec} thus rewrites as
\[ L (n, S^{- 1}) \cdot (u_n) = (b_n) \quad \text{where} \quad L = \sum_{k =
   0}^s p_k (X) Y^k \in R [X] [Y] . \]
When dealing with sequences that vanish eventually (or that converge fast
enough) as $n \rightarrow - \infty$, we can also consider operators of
infinite order
\[ p_0 (n) + p_1 (n) S^{- 1} + p_2 (n) S^{- 2} + \cdots = \sum_{i =
   0}^{\infty} p_i (n) S^{- i} = L (n, S^{- 1}) \]
where $L \in R [X] [[Y]]$.
\pagebreak[1]

In the same way as with recurrences, we view the multiplication by~$z$ of elements of $R [[z]]$ as a
linear operator that can be combined with the differentiation operator $\mathd
/ \mathd z$ to form linear differential operators with polynomial or
series coefficients. For example, we have
\[ \left( \frac{\mathd}{\mathd z}  \frac{1}{1 - z} \right) \cdot u (z) =
   \frac{\mathd}{\mathd z} \cdot \frac{u (z)}{1 - z} = \left( \frac{1}{1 - z} 
   \frac{\mathd}{\mathd z} + \frac{1}{(1 - z)^2} \right) \cdot u (z) \]
where the rational functions are to be interpreted as power series.

\begin{lemma}
  \label{lem:recdeq}Let $(u_n)_{n \in \mathbb{Z}}, (v_n)_{n \in \mathbb{Z}}$
  be sequences of elements of~$R$, with $u_n = v_n = 0$ for $n < 0$. Consider
  a recurrence operator of the form $L (n, S^{- 1})$ with $L (X, Y) \in R [X]
  [[Y]]$. The sequences $(u_n), (v_n)$ are related by the recurrence relation
  $L (n, S^{- 1}) \cdot (u_n) = (v_n)$ if and only if their generating series
  satisfy the differential equation
  \[ L \left( z \frac{\mathd}{\mathd z}, z \right) \cdot u (z) = v (z) . \]
\end{lemma}

\begin{proof}
  This follows from the relations
  \[ \sum_{n = - \infty}^{\infty} nf_n z^n = z \frac{\mathd}{\mathd z} \sum_{n
     = - \infty}^{\infty} f_n z^n, \qquad \sum_{n = - \infty}^{\infty} f_{n -
     1} z^n = z \sum_{n = - \infty}^{\infty} f_n z^n, \]
  noting that the operators of infinite order with respect to $S^{- 1}$ that
  may appear when the coefficients of the differential equation are series are
  applied to sequences that vanish for negative~$n$.
\end{proof}

Generating series of sequences satisfying recurrences of the
form~\eqref{eq:rec}---in other words, by Lemma~\ref{lem:recdeq}, formal series
solutions of linear differential equations with polynomial coefficients---are
called {\tmem{differentially finite}} or {\tmem{holonomic}}. We refer the
reader to~{\cite{Kauers2013,Salvy2019}} for an overview of the powerful
techniques available to manipulate these series and their generalizations to
several variables.

\section{Asymptotics}

One of the main appeals of generating series is the access they give
to the {\tmem{asymptotics}} of the corresponding sequences. The basic
fact here is simply the Cauchy-Hadamard theorem stating that the inverse of
the radius of convergence of $u (z)$ is the limit superior of $\abs{u_n}^{1 /
n}$ as $n \rightarrow \infty$. Concretely, as soon as we have an expression of
$u (z)$ (or an equation satisfied by it) that makes it clear that it has a
positive radius of convergence and where the complex singularities of the
corresponding analytic function are located, the exponential growth order
of~$\abs{u_n}$ follows immediately.

Much more precise results are available when more is known about the nature of
singularities. We quote here a simple result of this kind that will be enough
for our purposes, and refer to the book by Flajolet and
Sedgewick~{\cite{FlajoletSedgewick2009}} for far-ranging generalizations (see
in particular {\cite[Corollary~VI.1, p.~392]{FlajoletSedgewick2009}} for a
statement containing the following lemma as a special case).

\begin{lemma}
  \label{lem:asympt}Assume that, for some $\rho > 0$, the series $u (z) =
  \sum_n u_n z^n$ converges for $\abs{z} < \rho$ and that its sum has a single
  singularity $\alpha \in \mathbb{C}$ with $\abs{\alpha} = \rho$. Let $\Omega$
  denote a disk of radius $\rho' > \rho$, slit along the ray $\{ t \alpha \of
  t \in [\rho, \rho'] \}$, and assume that $u (z)$ extends analytically
  to~$\Omega$. If for some $C \in \mathbb{C}$ and $m \in
  \mathbb{C}\backslash\mathbb{Z}_{\leqslant 0}$, one has
  \[ u (z) \sim \frac{C}{(1 - \alpha^{- 1} z)^m} \]
  as $z \rightarrow \alpha$ from within~$\Omega$, then the corresponding
  coefficient sequence satisfies
  \[ u_n \sim \frac{C}{\Gamma (m)} n^{m - 1} \alpha^{- n} \]
  as $n \rightarrow \infty$,
  where $\Gamma$~is the Euler Gamma function.
\end{lemma}

\section{Majorant Series}\label{sec:maj}

While access to identities of sequences and to their asymptotic behavior is
important for error analysis, we are primarily interested in
{\tmem{inequalities}}. A natural way to express bounds on sequences encoded by
generating series is by majorant series, a classical idea of ``19th century''
analysis.

\begin{definition}
  \label{def:majorant}Let $f = \sum_n f_n z^n \in \mathbb{C} [[z]]$.
  \begin{enumerate}[(a)]
  \item A formal series with nonnegative coefficients
  $\hat{f} = \sum_n \hat{f}_n z \in \mathbb{R}_{\geqslant 0} [[z]]$ is said to
  be a {\tmem{majorant series}} of $f$ when we have $\abs{f_n} \leqslant
  \hat{f}_n$ for all~$n \in \mathbb{N}$. We then write $f \ll \hat{f}$.
  
  \item We denote by $\minmaj{f}$ the minimal majorant series of~$f$, that is,
  $\minmaj{f} = \sum_n \abs{f_n} z^n$.
  \end{enumerate}
\end{definition}

We also write $f \gg 0$ to indicate simply that $f$ has real, nonnegative
coefficients. Series denoted with a hat always have nonnegative coefficients,
and $\hat{f}$ is typically some kind of bound on~$f$, though not necessarily a
majorant series in the sense of the above definition.
While, for simplicity, we limit ourselves here to $f \in \mathbb{C}[[z]]$,
one can extend the definition to series~$f$ with coefficients in a normed algebra
(see Section~\ref{sec:wave}).

The following properties are classical and easy to check (see, e.g.,
Hille~{\cite[Section~2.4]{Hille1976}}).

\begin{lemma}
  \label{lem:maj-series}Let $f, g \in \mathbb{C} [[z]]$, $\hat{f}, \hat{g} \in
  \mathbb{R}_{\geqslant 0} [[z]]$ be such that $f \ll \hat{f}$ and $g \ll
  \hat{g}$.
  \begin{enumerate}
    \item \label{item:maj-series:basic}The following assertions hold, where
    $f_{N :} (z) = \sum_{n \geqslant N} f_n z^n$:
    \[ \begin{array}{lll}
         \text{(a)} \quad f + g \ll \hat{f} + \hat{g}, & \text{(b)} \quad
         \gamma f \ll \abs{\gamma}  \hat{f} \: \text{for $\gamma \in \mathbb{C}$},
         & \\
         \text{(c)} \quad f_{N :} (z) \ll \hat{f}_{N :} (z) \: \text{for $N \in
         \mathbb{N}$}, & \text{(d)} \quad f' (z) \ll \hat{f}' (z), & \\
         \text{(e)} \quad \left( \int_0^z f \right) \ll \left( \int_0^z
         \hat{f} \right), & \text{(f)} \quad fg \ll \hat{f}  \hat{g} . & 
       \end{array} \]
    \item \label{item:maj-series:compose}The disk of convergence~$\hat{D}$ of
    $\hat{f}$ is contained in that of~$f$, and when $\hat{g}_0 \in \hat{D}$,
    we have $f (g (z)) \ll \hat{f} (\hat{g} (z))$. In particular, $\abs{f (\zeta)
   }$ is bounded by $\hat{f} (\abs{\zeta})$ for all $\zeta \in \hat{D}$.
  \end{enumerate}
\end{lemma}

While majorant series are a concise way to express some types of inequalities
between sequences, their true power comes from Cauchy's {\tmem{method of
majorants}}~{\cite{Cauchy1841,Cauchy1842}}\footnote{See Cooke~{\cite{Cooke}}
for an interesting account of the history of this method and its extensions,
culminating in the Cauchy-Kovalevskaya theorem on partial differential
equations.}. This method is a way of computing majorant series of solutions of
functional equations that reduce to fixed-point equations. The idea is that
when the terms of a series solutions can be determined iteratively from the
previous ones, it is often possible to ``bound'' the equation by a simpler
``model equation'' whose solutions (with suitable initial values) then
automatically majorize those of the original equation.

A very simple result of this kind states that the solution~$y$ of a
{\tmem{linear}} equation $y = ay + b$ is bounded by the solution~$\hat{y}$ of
$\hat{y} = \hat{a}  \hat{y} + \hat{b}$ when $a \ll \hat{a}$, $b \ll \hat{b}$
and $\hat{a}_0 = 0$. Let us prove a variant of this fact. The previous
statement follows by applying the lemma to~$\minmaj{y}$.

\begin{lemma}
  \label{lem:maj linear}Let $\hat{a}, \hat{b}, y \in \mathbb{R}_{\geqslant 0}
  [[z]]$ be power series with $\hat{a}_0 = 0$ such that (note the $\ll$ sign)
  \[ y (z) \ll \hat{a} (z) y (z) + \hat{b} (z) . \]
  Then one has
  \[ y (z) \ll \hat{y} (z) \assign \frac{\hat{b} (z)}{1 - \hat{a} (z)} . \]
\end{lemma}

\begin{proof}
  Extracting the coefficient of $z^n$ on both sides of the inequality on~$y
  (z)$ yields
  \begin{equation}
    \abs{y_n} \leqslant \sum_{i = 1}^n \hat{a}_i  \abs{y_{n - i}} + \hat{b}_n,
    \label{eq:maj linear rec}
  \end{equation}
  where the sum starts at $i = 1$ due to the assumption that~$\hat{a}_0 = 0$.
  Similarly, the $\hat{y}_n$ satisfy
  \begin{equation}
    \hat{y}_n = \sum_{i = 1}^n \hat{a}_i  \hat{y}_{n - i} + \hat{b}_n .
    \label{eq:maj linear maj rec}
  \end{equation}
  We see by comparing \eqref{eq:maj linear rec}~and~\eqref{eq:maj linear maj
  rec} that $\abs{y_k} \leqslant \hat{y}_k$ for all $k < n$ implies $\abs{y_n}
  \leqslant \hat{y}_n$ (including the trivial case $\abs{y_0} \leqslant
  \hat{b}_0 = \hat{y}_0$) so that, by induction, one has $\abs{y_n} \leqslant
  \hat{y}_n$ for all~$n$.
\end{proof}

Another classical instance of the method applies to nonsingular linear
differential equations with analytic coefficients. In combination with
Lemma~\ref{lem:recdeq} above, it allows us to derive bounds on linear
recurrence sequences with polynomial coefficients.
Note that, since the correspondence described in Lemma~\ref{lem:recdeq}
maps~$S$ to~$z^{-1}$, Proposition~\ref{prop:maj deq} covers the case of
recurrences of infinite order.

\begin{proposition}
  \label{prop:maj deq}Let $a_0, \ldots, a_{r - 1}, b \in \mathbb{C} [[z]]$,
  $\hat{a}_0, \ldots, \hat{a}_{r - 1}, \hat{b} \in \mathbb{R}_{\geqslant 0}
  [[z]]$ be such that $a_k \ll \hat{a}_k$ for $0 \leqslant k < r$ and $b \ll
  \hat{b}$. Assume that $\hat{y} \in \mathbb{R} [[z]]$ is a solution of the
  equation
  \begin{equation}
    \hat{y}^{(r)} (z) - \hat{a}_{r - 1} (z)  \hat{y}^{(r - 1)} (z) - \cdots -
    \hat{a}_1 (z)  \hat{y}' (z) - \hat{a}_0 (z)  \hat{y} (z) = \hat{b} (z) .
    \label{eq:maj diff eq}
  \end{equation}
  Then, any solution~$y \in \mathbb{C} [[z]]$ of
  \begin{equation}
    y^{(r)} (z) - a_{r - 1} (z) y^{(r - 1)} (z) - \cdots - a_1 (z) y' (z) -
    a_0 (z) y (z) = b (z) \label{eq:diff eq}
  \end{equation}
  with $\abs{y_0} \leqslant \hat y_0, \ldots, \abs{y_{r - 1}} \leqslant
  \hat{y}_{r - 1}$ satisfies $y \ll \hat{y}$.
\end{proposition}

\begin{proof}
  Write $y^{(k)} (z) = \sum_n (n + k)^{\underline{k}} y_{n + k} z^n$, where
  $n^{\underline{k}} = n (n - 1) \cdots (n - k + 1)$. The equation on~$y (z)$
  translates into
  \[ \sum_n (n + r)^{\underline{r}} y_{n + r} z^n - \sum_{k = 0}^{r - 1}
     \sum_n \sum_{j = 0}^{\infty} a_{k, j}  (n + k)^{\underline{k}} y_{n + k -
     j} z^n = \sum_n b_n z^n, \]
  whence
  \[ n^{\underline{r}} y_n = \sum_{j = 0}^{\infty} \sum_{k' = 1}^r a_{k, j} 
     (n - k')^{\underline{r - k'}} y_{n - k' - j} + b_{n - r}, \]
  and similarly for $\hat{y} (z)$. As with Lemma~\ref{lem:recdeq}, these
  formulae hold for $n \in \mathbb{Z}$. The right-hand side only involves
  coefficients $y_j$ with $j < n$, and the polynomial coefficients $(n -
  k')^{\underline{r - k'}}$, including $n^{\underline{r}}$, are nonnegative as
  soon as $n \geqslant r$. For $n \geqslant r$ and assuming $\abs{y_k} \leqslant
  \hat{y}_k$ for all $k < n$, we thus have
  \begin{align*}
    n^{\underline{r}}  \abs{y_n} & \leqslant \sum_{j = 0}^{\infty} \sum_{k' =
    1}^r \abs{a_{k, j}}  (n - k')^{\underline{r - k'}}  \abs{y_{n - k' - j}} + |
    b_{n - r} |\\
    & \leqslant \sum_{j = 0}^{\infty} \sum_{k' = 1}^r \hat{a}_{k, j}  (n -
    k')^{\underline{r - k'}}  \hat{y}_{n - k' - j} + \hat{b}_{n - r}\\
    & = n^{\underline{r}}  \hat{y}_n .
  \end{align*}
  The result then follows by induction from the inequalities $|
  y_0 | \leqslant \hat y_0, \ldots, \allowbreak {\abs{y_{r - 1}} \leqslant \hat{y}_{r -
  1}}$.
\end{proof}

Like in the case of linear algebraic equations, this result admits variants
that deal with differential inequalities. We limit ourselves to first-order
equations here.

\begin{lemma}
  \label{lem:maj diff ineq}Consider power series $\hat{a}_0, \hat{a}_1, b, y
  \in \mathbb{R}_{\geqslant 0} [[z]]$ with nonnegative coefficients such that
  $\hat{a}_1 (0) = 0$ and
  \begin{equation}
    y' (z) \ll \hat{a}_1 (z) y' (z) + \hat{a}_0 (z) y (z) + \hat{b} (z) .
    \label{eq:diff ineq}
  \end{equation}
  The equation
  \begin{equation}
    \hat{y}' (z) = \hat{a}_1 (z)  \hat{y}' (z) + \hat{a}_0 (z)  \hat{y} (z) +
    \hat{b} (z) \label{eq:maj diff ineq}
  \end{equation}
  admits a unique solution $\hat{y}$ with $\hat{y} (0) = y (0)$, and one has
  $y \ll \hat{y}$.
\end{lemma}

\begin{proof}
  Since $\hat{a}_{1, 0} = 0$, the right-hand side of the inequality
  \[ (n + 1) y_{n + 1} \leqslant \sum_{j = 1}^n \hat{a}_{1, j}  (n - j + 1)
     y_{n - j + 1} + \sum_{j = 0}^n \hat{a}_{0, j} y_{n - j} + \hat{b}_n \]
  corresponding to the extraction of the coefficient of $z^n$ in~\eqref{eq:diff
  ineq} only involves coefficients $y_j$ with $j \leqslant n$. Equation
  \eqref{eq:maj diff ineq} corresponds to a recurrence of a similar shape (and
  therefore has a unique solution), and one concludes by comparing these
  relations.
\end{proof}

Solving majorant equations of the type~\eqref{eq:maj diff eq}, \eqref{eq:maj diff
ineq} yields majorants involving antiderivatives. The following observation
can be useful to simplify the resulting expressions.

\begin{lemma}
  \label{lem:ipp}For $\hat{f}, \hat{g} \in \mathbb{R}_{\geqslant 0} [[z]]$,
  one has $\int_0^z (\hat{f}  \hat{g}) \ll \hat{f}  \int_0^z \hat{g}$. In
  particular, $\int_0^z \hat{f}$ is bounded by~$z \hat{f} (z)$.
\end{lemma}

\begin{proof}
  Integration by parts shows that $\int_0^z (\hat{f}  \hat{g}) - \hat{f} 
  \int_0^z \hat{g} \in \mathbb{R}_{\geqslant 0} [[z]]$.
\end{proof}

It is possible to state much more general results along these lines, and
cover, among other things, general implicit functions, solutions of partial
differential equations, and various singular equations. We refer to
{\cite[Chap.~VII]{Cartan1961}}, {\cite{vanderHoeven2003}},
{\cite{WarneWarneSochackiParkerCarothers2006}},
and~{\cite[Appendix~A]{GiustiLecerfSalvyYakoubsohn2005}} for some
results that may turn out to be useful in more complicated error analyses.

\section{Floating-Point Errors}\label{sec:fp}

The toy example from Section~\ref{sec:toy abs} illustrates error propagation
in fixed-point arithmetic, where each elementary operation introduces a
bounded absolute error. In a floating-point setting, the need to deal with the
propagation of relative errors complicates the analysis. Thorough treatments
of the analysis of floating-point computations can be found in the books of
Wilkinson~{\cite{Wilkinson1963}} and Higham~{\cite{Higham2002}}. We will use
the following definitions and properties.

For simplicity, we assume that we are working in binary floating-point
arithmetic with unbounded exponents. We make no attempt at covering
underflows%
\footnote{It would be interesting to extend the methodology to
this case.
Doing so might require adapting the results of this section and the previous one to deal with equations mixing features of absolute and relative error analysis.}.
Following standard practice, our error bounds are mainly based on the
following inequalities that link the approximate version~$\tilde{\ast}$ of
each arithmetic operation $\ast \nosymbol \in \{ +, -, \times, / \}$ to the
corresponding exact mathematical operation: \medskip
\begin{center}
\begin{tabular}{llll}
  $x \tilde{\ast} y = (x \ast y)  (1 + \delta_1)$, & $\abs{\delta_1} \leqslant
  u$ & --- & ``first standard model'' \cite[e.g.,][(2.4)]{Higham2002},\\
  $x \tilde{\ast} y = (x \ast y)  (1 + \delta_2)^{- 1}$, & $\abs{\delta_2}
  \leqslant u$ & --- &  ``modified standard model'' \cite[e.g.,][(2.5)]{Higham2002}.
\end{tabular}
\medskip
\end{center}
In addition, we occasionally use the fact that multiplications by
powers of two are exact.

The quantity~$u$ that appears in the above bounds is called the {\tmem{unit
roundoff}} and depends only on the precision and rounding mode. For example,
in standard $t$\mbox{-}bit round-to-nearest binary arithmetic, one can take $u
= 2^{- t}$.

The two ``standard models'' are somewhat redundant, and most error analyses in
the literature proceed exclusively from the first standard model. However,
working under the modified standard model sometimes helps avoid assumptions
that $nu < 1$ when studying the effect of a chain of $n$~operations, which is
convenient when working with generating series.

Suppose that a quantity $x \neq 0$ is affected by successive relative errors
$\delta_1, \delta_2, \ldots$ resulting from a chain of dependent operations,
with $\abs{\delta_i} \leqslant u$ for all~$i$. The cumulative relative
error~$\eta$ after $n$~steps is given by
\[ 1 + \eta = \prod_{i = 1}^n (1 + \delta_i) . \]
This leads us to introduce the following notation.
Definition~\ref{def:theta-hat}(\ref{item:theta}) below is
adapted from Higham's notation~\cite[Chapter~3]{Higham2002}, but more restrictive
compared with the assumptions of Lemmas 3.1~and~3.3 in~\cite{Higham2002}.

\begin{definition}
  \label{def:theta-hat}When the roundoff error~$u$ is fixed and clear from the
  context:
  \begin{enumerate}[(a)]
  \item \label{item:theta}
  We write $\eta = \theta_n$ to indicate that $1 +
  \eta = \prod_{i = 1}^n (1 + \delta_i)$ for some $\delta_1, \ldots, \delta_n$
  with $\abs{\delta_i} \leqslant u$.
  
  \item We define $\hat{\theta}_n = (1 + u)^n - 1$ for $n \geqslant 0$. As
  usual, this sequence is extended to $n \in \mathbb{Z}$ by setting
  $\hat{\theta}_n = 0$ for $n < 0$, and we also consider its generating series
  \[ \hat{\theta} (z) = \frac{1}{1 - (1 + u) z} - \frac{1}{1 - z} . \]
  \item More generally, we set
  \[ \hat{\theta}^{(p, q)} (z) = \sum_{n = 0}^{\infty} \hat{\theta}_{pn + q}
     z^n = \frac{(1 + u)^q}{1 - (1 + u)^p z} - \frac{1}{1 - z} . \]
  \end{enumerate}
\end{definition}

Thus, a series whose coefficient of index~$n$ is of the form~$\theta_n$
satisfies
\begin{equation}
  \sum_{n=0}^{\infty} \theta_n z^n \ll \hat{\theta} (z) . \label{eq:theta hat}
\end{equation}
Similarly, the generating series corresponding to a regularly spaced
subsequence of~$(\theta_n)$ (e.g., cumulative errors after every second operation)
is bounded by~$\hat{\theta}^{(p, q)}$ for appropriate $p$~and~$q$.
The inequality~\eqref{eq:theta hat} is closely related to that between
quantities $\abs{\theta_n}$ and $\gamma_n = nu / (1 - nu)$ used extensively in
Higham's book, and one has $\hat{\theta}_n \leqslant \gamma_n$ for $nu < 1$.
However, compared to {\cite[Lemma~3.1]{Higham2002}}, our definition
of~$\theta_n$ only allows for nonnegative powers of $1 + \delta_i$,
and \eqref{eq:theta hat} would fail to hold without this restriction.

To rewrite the relation $\tilde{x}_n = x_n  (1 + \theta_n)$ in terms of
generating series, we can use the {\tmem{Hadamard product}} of series, defined
by
\[ (f \odot g) (z) = \sum_n f_n g_n z^n . \]
An immediate calculation starting from Definition~\ref{def:theta-hat} yields a
closed-form expression of the Hadamard product with~$\hat{\theta}^{(p, q)}$.

\begin{lemma}
  \label{lem:hadamard theta hat}For any power series~$\hat{f} (z)$ with
  nonnegative coefficients, it holds that
  \[ (\hat{\theta}^{(p, q)} \odot \hat{f}) (z) = \sum_{n = 0}^{\infty}
     \hat{\theta}_{pn + q}  \hat{f}_n z^n = (1 + u)^q  \hat{f} ((1 + u)^p z) -
     \hat{f} (z) . \]
\end{lemma}

\section{Variable Coefficients: Legendre Polynomials}\label{sec:legendre}

With this background in place, we now consider a ``real'' application
involving a recurrence with polynomial coefficients (that, additionally, depend on a
parameter~$x$). The example is adapted%
\footnote{Proposition~\ref{prop:rec-error} and its proof are identical, up to
presentation details, to {\cite[Proposition~5]{JohanssonMezzarobba2018}},
and included here for expository reasons only.
The work eventually leading to the present paper actually started first
and found an unexpected application in~\cite{JohanssonMezzarobba2018},
which motivated us to develop it further.}
from material previously published
in~{\cite{JohanssonMezzarobba2018}}. Let $P_n$ denote the Legendre
polynomial of index~$n$, defined by
\[ \sum_{n = 0}^{\infty} P_n (x) z^n = \frac{1}{\sqrt{1 - 2 xz + z^2}} . \]
Fix $x \in [- 1, 1]$, and let $p_n = P_n (x)$. The classical three-term
recurrence
\begin{equation}
  p_{n + 1} = \frac{1}{n + 1}  ((2 n + 1) xp_n - np_{n - 1}), \qquad n \geqslant 0
  \label{eq:legendre exact}
\end{equation}
allows us to compute~$p_n$ for any~$n$ starting from $p_0 = 1$ and an
arbitrary $p_{- 1} \in \mathbb{R}$. Suppose that we run this computation in
fixed-point arithmetic, with an absolute error~$\varepsilon_n$ at step~$n$, in
the sense that the computed values~$\tilde{p}_n$ satisfy
\begin{equation}
  \tilde{p}_{n + 1} = \frac{1}{n + 1}  ((2 n + 1) x \tilde{p}_n - n
  \tilde{p}_{n - 1}) + \varepsilon_n, \qquad n \geqslant 0. \label{eq:legendre
  pert}
\end{equation}
An analysis not taking into account the dependencies between the errors at each
step yields $\abs{\tilde p_n - p_n} \lessapprox (1 + \sqrt2)^n \bar\varepsilon/4$
where $\bar\varepsilon = \max_n \varepsilon_n$.

\begin{proposition}
  \label{prop:rec-error}Let $(\tilde{p}_n)_{n \geqslant - 1}$ be a sequence of real
  numbers satisfying~\eqref{eq:legendre pert}, with $\tilde{p}_0 = 1$. Assume
  that $\abs{\varepsilon_n} \leqslant \bar{\varepsilon}$ for all~$n$. Then, for all
  $n \geqslant 0$, the global absolute error satisfies
  \[ \abs{\tilde{p}_n - p_n} \leqslant \frac{(n + 1)  (n + 2)}{4}  \bar{\varepsilon}
     . \]
\end{proposition}

\begin{proof}
  Let $\delta_n = \tilde{p}_n - p_n$ and $\eta_n = (n + 1) \varepsilon_n$.
  Subtracting \eqref{eq:legendre exact}~from~\eqref{eq:legendre pert} gives
  \begin{equation}
    \label{eq:rec-error} (n + 1) \delta_{n + 1} = (2 n + 1) x \delta_n - n
    \delta_{n - 1} + \eta_n,
  \end{equation}
  with $\delta_0 = 0$. Note that~\eqref{eq:rec-error} holds for all $n \in
  \mathbb{Z}$ if the sequences $(\delta_n)$ and $(\eta_n)$ are extended by~$0$
  for $n < 0$. By Lemma~\ref{lem:recdeq}, it translates into
  \[ (1 - 2 xz + z^2) z \frac{\mathd}{\mathd z} \delta (z) = z (x - z) \delta
     (z) + z \eta (z) . \]
  The solution of this differential equation with $\delta (0) = 0$ reads
  \[ \delta (z) = p (z)  \int_0^z \eta (w)  \hspace{0.17em} p (w)
     \hspace{0.17em} \mathd w, \qquad p (z) = \sum_{n = 0}^{\infty} p_n z^n =
     \frac{1}{\sqrt{1 - 2 xz + z^2}} . \]
  It is well known that $\abs{P_n}$ is bounded by~$1$ on $[- 1, 1]$, so that $p
  (z) \ll (1 - z)^{- 1}$, and the definition of~$\eta_n$ implies $\eta (z) \ll
  \bar{\varepsilon}  (1 - z)^{- 2}$. It follows by Lemma~\ref{lem:maj-series}
  that
  \[ \delta (z) \ll \frac{1}{1 - z}  \int_0^z \frac{\bar{\varepsilon}}{(1 -
     w)^3} \hspace{0.17em} \mathd w = \frac{\bar{\varepsilon}}{2 (1 - z)^3} \]
  and therefore $\abs{\delta_n} \leqslant (n + 1)  (n + 2)  \bar{\varepsilon} / 4.$
\end{proof}

Reasoning in the same way but using the inequality
$p(z) \ll \sqrt2 (1-x^2)^{-1/4} (1-z)^{-1/2}$
\cite[e.g.,][proof of Proposition~3]{JohanssonMezzarobba2018}
instead of $p(z) \ll (1-z)^{-1}$, we can also prove that one has
$|\tilde p_n - p_n| < \frac43 (1-x^2)^{-1/2} n \bar\varepsilon$
when $|x| < 1$.
For comparison, Hrycak and Schmutzhard~\cite{HrycakSchmutzhard2018} show that
when \eqref{eq:legendre exact} is evaluated in \emph{floating-point}
arithmetic with unit roundoff~$u$,
the absolute error~$|\tilde p_n - p_n|$ is bounded
by~$21 n^2 u$,
and by $129 (1-x^2)^{-1/2} n u$ for $\abs{x} < 1$,
in both cases assuming that $5 n u^{1/2} \leqslant 1$.

\section{Relative Errors: The Toy Example Revisited}\label{sec:toy rel}

Let us return to the recurrence
\begin{equation}
  c_{n + 1} = 2 c_n - c_{n - 1} \label{eq:toy exact bis}
\end{equation}
considered in Section~\ref{sec:toy abs}, but now look at what happens when the
computation is carried out in floating-point arithmetic, using the observations
made in Section~\ref{sec:fp}.

We assume binary floating-point arithmetic with unit roundoff~$u$, and
consider the iterative computation of the sequence defined by~\eqref{eq:toy
exact bis} with $c_{- 1} = 0$ and (for simplicity) an exactly representable initial
value~$c_0$.
We have $c_n = (n + 1) c_0$, that is,
\[ c (z) = \frac{c_0}{(1 - z)^2} . \]
The floating-point computation produces a sequence of
approximations $\tilde{c}_n \approx c_n$ with $\tilde{c}_0 = c_0$. Using the
standard model recalled in Section~\ref{sec:fp} and the fact that
multiplication by~$2$ is exact, the analogue of the local estimate~\eqref{eq:toy
fxp} reads
\begin{equation}
  \tilde{c}_{n + 1} = (2 \tilde{c}_n - \tilde{c}_{n - 1})  (1 +
  \varepsilon_n), \qquad \abs{\varepsilon_n} \leqslant u. \label{eq:toy fp}
\end{equation}
Let $\delta_n = \tilde{c}_n - c_n$. By subtracting $(1 + \varepsilon_n)$
times~\eqref{eq:toy exact bis} from \eqref{eq:toy fp} and reorganizing, we get
\[ \delta_{n + 1} - 2 \delta_n + \delta_{n - 1} = \varepsilon_n  (c_{n + 1} +
   2 \delta_n - \delta_{n - 1}), \]
which rewrites
\[ (z^{- 1} - 2 + z) \delta (z) = \varepsilon (z) \odot (z^{- 1} c (z) + (2 -
   z) \delta (z)) . \]
We multiply this equation by~$z$ to get
\[ (1 - z)^2 \delta (z) = (z \varepsilon (z)) \odot (c (z) + z (2 - z) \delta
   (z)) . \]
Denote $\gamma (z) = (1 - z)^2 \delta (z)$ and $b (z) = (1 - z)^{- 2} - 1$.
Since $\abs{\varepsilon_n} \leqslant u$ and $b (z) \gg 0$, we have
\[ \gamma (z) = (z \varepsilon (z)) \odot (c (z) + b (z) \gamma (z)) \ll u
   \left( \minmaj{c} (z) + b (z)  \minmaj{\gamma} (z) \right), \]
and therefore
\[ \minmaj{\gamma} (z) \ll u \left( \minmaj{c} (z) + b (z)  \minmaj{\gamma}
   (z) \right) \]
where $b (0) = 0$. By Lemma~\ref{lem:maj linear}, it follows that
\[ \gamma (z) \ll \minmaj{\gamma} (z) \ll \minmaj{c} (z)  \frac{u}{1 - ub (z)}
   = \frac{\abs{c_0}}{(1 - z)^2}  \frac{(1 - z)^2 u}{1 - 2 (1 + u) z + (1 + u)
   z^2}, \]
whence
\begin{equation}
  \delta (z) \ll \frac{\abs{c_0}}{(1 - z)^2}  \frac{u}{(1 - \alpha z)  (1 -
  \beta z)} \backassign \hat{\delta} (z) \label{eq:toy fp first global bound}
\end{equation}
where $\alpha > \beta$ are the roots of $z^2 - 2 (1 + u) z + (1 + u)$.

From~\eqref{eq:toy fp first global bound}, a trained eye immediately reads off
the essential features the bound. Perhaps the most important information is
its asymptotic behavior as the working precision increases. For a bound that
depends on a problem dimension~$n$, it is customary to focus (sometimes
implicitly) on the leading order term as $u \rightarrow 0$ for
fixed\footnote{See Higham's book~{\cite{Higham2002}} for many
examples, and in particular the discussion of linearized bounds at the
beginning of Section~3.4. Already in equation~3.7, the implicit assumption $nu
< 1$ is not sufficient to ensure that the neglected term is $\Omicron (u^2)$,
if $n$ is allowed to grow while $u$ tends to zero.}~$n$, and in some cases to
further simplify it by looking at its asymptotic behavior for large~$n$.

In the present case, the definition of~$\alpha$ as a root of $z^2 - 2 (1 + u)
z + (1 + u)$ yields $\alpha = 1 + u^{1 / 2} + \Omicron (u)$, and hence
\begin{equation}
  \hat{\delta} (z) = \frac{\abs{c_0} u}{(1 - z)^4} + \Omicron (u^{3 / 2}) \qquad
  \text{as $u \rightarrow 0$}, \label{eq:toy fp asympt}
\end{equation}
or equivalently
\begin{equation}\label{eq:toy fp asympt 2}
  \hat{\delta}_n = \frac{1}{6}  (n + 1)  (n + 2)  (n + 3)  \abs{c_0} u +
   \Omicron (u^{3 / 2}), \qquad u \rightarrow 0, \quad \text{$n$ fixed} .
\end{equation}
The fact that $\hat{\delta}_n \approx n^3  \abs{c_0} u / 6$ in the sense that
$\lim_{u \rightarrow 0}  (u^{- 1}  \hat{\delta}_n) \sim_{n \rightarrow \infty}
n^3  \abs{c_0} u / 6$ also follows directly from~\eqref{eq:toy fp asympt} using
Lemma~\ref{lem:asympt}.
(One also sees that $\hat{\delta}_n = C \alpha^n +
\Omicron (n)$ as $n \rightarrow \infty$ for fixed~$u$, where $C$~could easily
be made explicit. This is however less relevant for our purposes,
cf.~Remark~\ref{rk:toy parfrac} below.)

As observed by one of the referees, \eqref{eq:toy exact bis}~is a special case of
the classical three-term recurrence for Chebyshev polynomials.
One has $c_n = c_0 \, U_n(1)$ where $U_n$~is the Chebyshev polynomial of the second kind.
One can thus view~$c_n$ as the value at~$1$ of a Chebyshev series reduced to a single term and compare~\eqref{eq:toy fp asympt 2} with known error bounds for the floating-point evaluation of Chebyshev series.
For example, the bound from \cite[Theorem~6]{Barrio2002} specializes in our setting to
$\abs{\delta_n} \leqslant \frac43 (n+1)(n^2+\frac{11}{4}n+3) \abs{c_0} u + \Omicron(u^2)$
and is hence just slightly worse than~\eqref{eq:toy fp asympt 2}, though much more general.

In order to obtain a bound that holds for all $n$~and~$u$, we can majorize both $(1 - z)^{-
1}$ and $(1 - \beta z)^{- 1}$ by $(1 - \alpha z)^{- 1}$ in~\eqref{eq:toy fp
first global bound}. We obtain $\hat{\delta} (z) \ll \abs{c_0} u (1 - \alpha
z)^{- 4}$, and therefore
\begin{equation}
  \abs{\delta_n} \leqslant \hat{\delta}_n \leqslant \abs{c_0}  \frac{(n + 1) 
  (n + 2)  (n + 3)}{6} \alpha^n u. \label{eq:toy bound}
\end{equation}
Recalling that $c_n = c_0  (n + 1)$, we conclude that $\tilde{c}_n = c_n  (1 +
\eta_n)$ where
\begin{equation} \label{eq:toy bound 2}
\abs{\eta_n} \leqslant \frac{(n + 2)  (n + 3)}{6} \alpha^n u.
\end{equation}
For an even more precise bound, one could also isolate the leading term of the
series expansion of~$\hat{\delta} (z)$ with respect to~$u$ and reason as above
to conclude that
\begin{equation} \label{eq:toy first order}
\abs{\eta_n} \leqslant \frac{1}{6}  {(n + 2)}  {(n + 3)} u + p (n) \alpha^n u^2
\end{equation}
for some explicit polynomial~$p (n)$.

With no other assumption on the error of subtraction
than the first standard model of floating-point arithmetic, the
bound~\eqref{eq:toy fp first global bound} is sharp: when $c_0 \geqslant 0$
and $\varepsilon_n \equiv u$, we have $\delta (z) = \hat{\delta} (z)$.
This means the exponential growth with~$n$ for fixed~$u$ is unavoidable
under these hypotheses.
However, the
exponential factor only starts contributing significantly when $n$~becomes
extremely large compared to the working precision. It is natural to control it
by tying the growth of~$n$ to the decrease of~$u$. In particular, it is clear
that $\alpha^n = \Omicron (1)$ if $n = \Omicron (u^{- 1 / 2})$. One can check
more precisely that $\alpha^n \leqslant e$ as long as $n \leqslant (\alpha -
1)^{- 1} \approx u^{- 1 / 2}$, and $\alpha^n \leqslant 3$ for all $n \leqslant
u^{- 1 / 2}$ provided that $u \leqslant 2^{- 7}$.
The leading term on the right-hand side
of~\eqref{eq:toy first order} is optimal as well, for the same reason.

\begin{figure}
  \centerline{\includegraphics{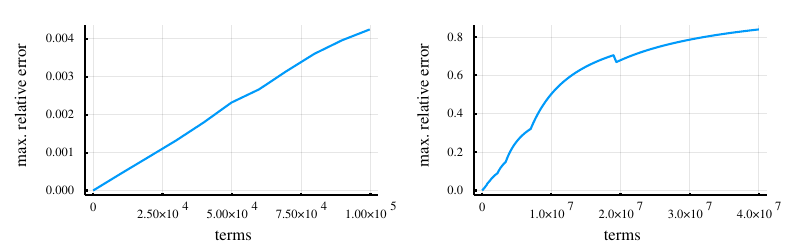}}
  \vspace*{-.5\baselineskip}
  \caption{%
    Measured error in the evaluation of the sequence~$(c_n)$ defined
    by~\eqref{eq:toy exact bis} in IEEE binary32 floating-point arithmetic
    ($u = 2^{-24} \simeq 6 \cdot 10^{-8}$).
    For each~$n$, we plot the maximum observed
    error on~$c_n$ for 997 exactly representable values of~$c_0$ regularly spaced
    in the interval~$[\pi/2, \pi]$.
    The script used to produce these plots is available in the supplementary material.
  }
  \label{fig:toy}
\end{figure}

Figure~\ref{fig:toy} illustrates, on two different scales, the accumulation of errors in numerical experiments.
We can see that the relative error actually growths roughly \emph{linearly} until it saturates around~$1$ (meaning that the computed results no longer have a single correct significant digit but remain of the correct order of magnitude).
Thus the bounds \eqref{eq:toy bound 2}, \eqref{eq:toy first order}
are still very pessimistic compared to reality.
This was to be expected: indeed, it would be quite surprising for all local errors~$\varepsilon_n$ to align to reach the worst case rather than be more or less evenly distributed in $[-u, u]$.
The standard model does not capture this fact.
We refer to Higham~\cite[Section~2.8]{Higham2002} and Higham and Mary~\cite{HighamMary2019} for a discussion of the kinds of bounds that can be derived if one is willing to make assumptions on the distribution of individual rounding errors.

\begin{remark}
  \label{rk:toy parfrac}Starting back from~\eqref{eq:toy fp first global bound},
  one may be tempted to write the partial fraction decomposition of
  $\hat{\delta} (z)$ and deduce an exact formula for $\hat{\delta}_n$. Doing
  so leads to
  \[ \hat{\delta}_n = (A \alpha^n - B \beta^n - u^{- 1} n)  \abs{c_0} u, \qquad
     A = \frac{\alpha^3}{(\alpha - 1)^2  (\alpha - \beta)}, \quad B = A - 1 -
     u^{- 1} . \]
  This expression is misleading, for it involves cancellations between terms
  that tend to infinity as $u$~goes to zero. In particular, we have $A \sim
  \frac{1}{2} u^{- 3 / 2}$ as $u \rightarrow 0$.
\end{remark}

\section{Relative Errors, Infinite Order: Scaled Bernoulli
Numbers}\label{sec:bernoulli}

For a more sophisticated application of the same idea, let us study the
floating-point computation of scaled Bernoulli numbers as described by
Brent~{\cite[Section~7]{Brent1980}} (see also Brent and
Zimmermann~{\cite[Section~4.7.2]{BrentZimmermann2010}}). This section is the
first where not only the method but the results are new.

The scaled Bernoulli numbers are defined in terms of the classical Bernoulli
numbers~$B_k$ by $b_k = B_{2 k} / (2 k) !$. Their generating series has a
simple explicit expression:
\begin{equation}
  b (z) = \sum_{k = 0}^{\infty} b_k z^k = \frac{\sqrt{z} / 2}{\tanh \left(
  \sqrt{z} / 2 \right)} . \label{eq:bernoulli gen-b}
\end{equation}
One possible algorithm for computing~$b_k$, suggested by Reinsch according to
Brent~{\cite[Section~12]{Brent1980}}, is to follow the recurrence
\begin{equation}
  b_k = \frac{1}{(2 k) !4^k} - \sum_{j = 0}^{k - 1} \frac{b_j}{(2 k + 1 - 2 j)
  !4^{k - j}} . \label{eq:bernoulli exact rec}
\end{equation}
In {\cite[Exercise~4.35]{BrentZimmermann2010}}, it is asked to ``prove (or
give a plausibility argument for)'' the fact that the relative error on~$b_k$
when computed using~\eqref{eq:bernoulli exact rec} in floating-point arithmetic
is~$\Omicron (k^2 u)$. The $\Omicron (k^2 u)$ bound is already mentioned
without proof in~{\cite{Brent1980}}, and again
in~{\cite[Section~2]{BrentHarvey2013}}. Paul Zimmermann (private
communication, June~2018) suggested that the dependency in~$k$ may actually be
linear rather than quadratic.

Our goal in this section is to prove a version of the latter conjecture. Like
in the previous section, it cannot be true if the $\Omicron (\cdot)$ is
interpreted as uniform as $u \rightarrow 0$ and $k \rightarrow \infty$
independently. It does hold, however, when $u$~and~$k$ are restricted to a
region where their product is small enough, as well as in the sense that the
relative error~$\eta_k$ for fixed~$k$ satisfies $\abs{\eta_k} \leqslant C_k u$
when~$u$ is small enough, for a sequence $C_k$ which itself satisfies $C_k =
\Omicron (k)$ as $k \rightarrow \infty$. We will in fact derive a fully
explicit, non-asymptotic bound in terms of $u$ and~$k$.

Based on the form of~\eqref{eq:bernoulli gen-b}, denote $w = \sqrt{z} / 2$,
and for any $f \in \mathbb{C} [[z]]$, define $f^{\ast}$ by $f^{\ast} (w) = f
(4 w^2)$, so that $f^{\ast} (w) = f (z)$. In particular, we have $b^{\ast} (w)
= w / \tanh w$. We will use the following classical facts about the numbers $|
b_k |$.

\begin{lemma}
  \label{lem:bernoulli abs}The absolute values of the scaled Bernoulli numbers
  satisfy
  \[ \minmaj{b}^{\ast} (w) = 2 - \frac{w}{\tan w}, \qquad \abs{b_k} \sim_{k
     \rightarrow \infty} \frac{2}{(2 \pi)^{2 k}}, \qquad \frac{2}{(2 \pi)^{2
     k}} \leqslant \abs{b_k} \leqslant \frac{4}{(2 \pi)^{2 k}}, \]
  where the last formula assumes $k \geqslant 1$.
\end{lemma}

\begin{proof}
  The expression of $\minmaj{b}^{\ast} (w)$ can be deduced from that of
  $b^{\ast} (w)$ and the fact that $B_{2 n}$ has sign $(- 1)^{n + 1}$ for $n
  \geqslant 1$, using the relation $\tanh (iw) = i \tan (w)$. The other
  statements follow from the expression of Bernoulli numbers using the Riemann
  zeta function~{\cite[e.g.,][formula (25.6.2)]{DLMF}}.
\end{proof}

Let $\tilde{b}_k$ denote the approximate value of~$b_k$ computed
using~\eqref{eq:bernoulli exact rec}. We assume that the computed value of $n!$
is equal to $n! / (1 + q_n)$ with
$q_n = \theta_{n - 2}$, in the notation of
Section~\ref{sec:fp}, for $n\geqslant 2$,
and $q_0 = q_1 = 1$.
According to the modified standard model of
floating-point arithmetic, this holds true if $n!$ is computed as $(((2 \times
3) \times 4) \times \cdots \times n)$ and the working precision is at least
$\lceil \log_2 n \rceil$, even with no special treatment of multiplications by
powers of two. The local error introduced by one step of the
iteration~\eqref{eq:bernoulli exact rec} then behaves as follows.

\begin{lemma}
  \label{lem:bernoulli local}At every step
  of the iteration~\eqref{eq:bernoulli exact rec},
  the computed value~$\tilde{b}_k$ has the form
  \begin{equation}
    \tilde{b}_k = \frac{1 + s_k}{(2 k) !4^k} - \sum_{j = 0}^{k - 1}
    \frac{\tilde{b}_j  (1 + t_{k, j})}{(2 k + 1 - 2 j) !4^{k - j}}, \qquad
    \abs{s_k} \leqslant \hat{\theta}_{2 k}, \quad \abs{t_{k, j}} \leqslant
    \hat{\theta}_{3 (k - j) + 2} . \label{eq:bernoulli approx rec}
  \end{equation}
\end{lemma}

\begin{proof}
  The computation of $b_0$ ($=1$) involves no rounding error.
  Assume $k \geqslant 1$, and
  first consider the term~$a_k = 1 / ((2 k) !4^k)$ outside the sum. By
  assumption, the computed value of $(2 k) !$ is $(2 k) ! / (1 + q_{2 k})$,
  and the multiplication by $4^k$ that follows is exact. Inverting the result
  introduces an additional rounding error. The computed value of the whole
  term is hence $a_k  (1 + r_k')$ where $r_k' = \theta_{2 k - 1}$. By the same
  reasoning, the term of index~$j$ in the sum is computed with a relative
  error~$r_{k, j} = \theta_{2 (k - j) + 1}$ for all $j, k$. If $v_{k, i}$
  denotes the relative error in the addition of the term of index~$i$ to the
  partial sum for $0 \leqslant j \leqslant i - 1$, the computed value of the
  sum is hence
  \[ \sum_{j = 0}^{k - 1} \frac{\tilde{b}_j  (1 + r_{k, j})}{(2 k + 1 - 2 j)
     !4^{k - j}}  \prod^{k - 1}_{i = j} (1 + v_{k, i}) . \]
  Taking into account the relative error $v'_k$ of the final subtraction leads
  us to~\eqref{eq:bernoulli approx rec}, with
  \[ 1 + s_k = (1 + r'_k)  (1 + v'_k), \qquad 1 + t_{k, j} = (1 + r_{k, j}) 
     (1 + v'_k)  \prod^{k - 1}_{i = j} (1 + v_{k, i}),
     \]
  which concludes the proof.
  \pagebreak[1]
\end{proof}

Let $\delta_k = \tilde{b}_k - b_k$. Comparison of \eqref{eq:bernoulli approx
rec}~with~\eqref{eq:bernoulli exact rec} yields
\[ \delta_k = \frac{s_k}{(2 k) !4^k} - \sum_{j = 0}^{k - 1} \frac{\delta_j +
   \tilde{b}_j t_{k, j}}{(2 k + 1 - 2 j) !4^{k - j}}, \]
which rearranges into
\begin{equation}
  \sum_{j = 0}^k \frac{\delta_j}{(2 k + 1 - 2 j) !4^{k - j}} = \frac{s_k}{(2
  k) !4^k} - \sum_{j = 0}^{k - 1} \frac{(b_j + \delta_j) t_{k, j}}{(2 k + 1 -
  2 j) !4^{k - j}} . \label{eq:bernoulli err rec}
\end{equation}
Using the bounds from Lemma~\ref{lem:bernoulli local}, with \ $\hat{\theta}_{3
(k - j) + 2}$ replaced by $\hat{\theta}_{4 (k - j) + 2}$ to obtain slightly
simpler expressions later, it follows that
\begin{equation}
  \left| \sum_{j = 0}^k \frac{\delta_j}{(2 k + 1 - 2 j) !4^{k - j}} \right|
  \leqslant \frac{\hat{\theta}_{2 k}}{(2 k) !4^k} + \sum_{j = 0}^{k - 1}
  \frac{(\abs{b_j} + \abs{\delta_j})  \hat{\theta}_{4 (k - j) + 2}}{(2 k + 1 - 2
  j) !4^{k - j}} . \label{eq:bernoulli err ineq}
\end{equation}
Let us introduce the auxiliary series
\[ C (z) = \sum_{k = 0}^{\infty} \frac{z^k}{(2 k) !4^k} = \cosh w, \quad S (z)
   = \sum_{k = 0}^{\infty} \frac{z^k}{(2 k + 1) !4^k} = \frac{\sinh w}{w},
   \quad \check{S} (z) = \frac{w}{\sin w}, \]
\[ \tilde{C} (z) = \hat{\theta}^{(2, 0)} (z) \odot C (z), \quad \tilde{S} (z)
   = \hat{\theta}^{(4, 2)} (z) \odot (S (z) - 1) . \]
The inequality~\eqref{eq:bernoulli err ineq} (note that the sum on the
right-hand side stops at $k - 1$) translates into
\[ \delta (z) S (z) \ll \tilde{C} (z) + \tilde{S} (z)  \left( \minmaj{b} (z) +
   \minmaj{\delta} (z) \right) . \]
Since $\check{S} (z)$ has nonnegative coefficients~{\cite[(4.19.4)]{DLMF}}, we
have $S (z)^{- 1} = iw / \sin (iw) \ll \check{S} (z)$, hence
\[ \delta (z) \ll \check{S} (z)  \tilde{C} (z) + \check{S} (z)  \tilde{S} (z) 
   \minmaj{b} (z) + \check{S} (z)  \tilde{S} (z)  \minmaj{\delta} (z) . \]
As $\tilde{S} (z) = \Omicron (z)$, Lemma~\ref{lem:maj linear} applies and
yields
\begin{equation}
  \delta (z) \ll \hat{\delta} (z) \assign \frac{\check{S} (z)  \tilde{C} (z) +
  \check{S} (z)  \tilde{S} (z)  \minmaj{b} (z)}{1 - \check{S} (z)  \tilde{S}
  (z)} . \label{eq:bernoulli ineq sol}
\end{equation}
Using Lemma~\ref{lem:hadamard theta hat} to rewrite the Hadamard products, we
have
\begin{align}
  \tilde{C} (z) &= C (a^2 z) - C (z) = \cosh (aw) - \cosh (w),
  \label{eq:benoulli Ct} \\
  \tilde{S} (z) &= a^2 S (a^4 z) - S (z) - (a^2 - 1) = \frac{\sinh (a^2 w) -
  \sinh (w)}{w} - (a^2 - 1), \label{eq:bernoulli St}
\end{align}
where $a = 1 + u$. In addition, Lemma~\ref{lem:bernoulli abs} gives a formula
for~$\minmaj{b} (z)$. Thus \eqref{eq:bernoulli ineq sol} yields an explicit, if
complicated, majorant series for $\delta (z)$.

It is not immediately clear how to extract a readable bound on $\delta_k$ in
the style of \eqref{eq:toy bound}. However, we can already prove an asymptotic
version of Zimmermann's observation. The calculations leading to Propositions
\ref{prop:bernoulli asympt}~and~\ref{prop:bernoulli main} can be checked with
the help of a computer algebra system. A~worksheet that illustrates how to do
that using Maple is provided in the supplementary material.

\begin{proposition}
  \label{prop:bernoulli asympt}When $k$~is fixed and for small enough~$u$, the
  relative error $\eta_k = \delta_k / b_k$ satisfies $\abs{\eta_k} \leqslant C_k
  u$ for some~$C_k$. In addition, the constants $C_k$ can be chosen such that
  $C_k = \Omicron (k)$ as $k \rightarrow \infty$.
\end{proposition}

\begin{proof}
  As $u$ tends to zero, we have
  \begin{align*}
    \hat{\delta} (z) & = \left( \frac{2 (1 - \cosh w) \cos (w)}{w^{- 2} \sin
    (w)^2} + \frac{4 (\cosh w - 1) + w \sinh w}{w^{- 1} \sin w} \right) u +
    \Omicron (u^2)\\
    & \backassign \hat{\xi} (z) u + \Omicron (u^2),
  \end{align*}
  where the coefficients are to be interpreted as formal series in~$z$. Hence,
  for fixed~$k$, the error~$\delta_k$ is bounded by $\hat{\xi}_k u + \Omicron
  (u^2)$ as $u \rightarrow 0$. The function $\hat{\xi}^{\ast} (w)$ is
  meromorphic, with double poles at $w = \pm \pi$, corresponding for
  $\hat{\xi} (z)$ to a unique pole of minimal modulus (also of order two) at
  $z = 4 \pi^2$. This implies (by Lemma~\ref{lem:asympt}) that $(2 \pi)^{2 k} 
  \hat{\xi}_k = \Omicron (k)$ as $k \rightarrow \infty$. The result follows
  using the growth estimates from Lemma~\ref{lem:bernoulli abs}.
\end{proof}

In other words, there exist a constant $A$ and a function~$R$ such that
$\eta_k \leqslant Aku + R (k, u)$, where $R (k, u) = o (u)$ as~$u \rightarrow
0$ for fixed~$k$, but $R (k, u)$ might be unbounded if $k$ tends to infinity
while $u$ tends to zero. (Remark~\ref{rk:bernoulli exp} below shows that an
exponential dependency in~$k$ is in fact unavoidable.) To get a bound valid
for all $u$~and~$k$ in a reasonable region, let us study the denominator
of~\eqref{eq:bernoulli ineq sol} more closely. A similar argument
applied to $\hat{\xi}$ in the place of~$\hat{\delta}$ would make the
constant~$A$ explicit.

\begin{lemma}
  \label{lem:bernoulli dom root}For small enough~$u \geqslant 0$, the function
  $h : w \mapsto \tilde{S}^{\ast} (w) - \check{S}^{\ast} (w)^{- 1}$ has
  exactly two simple zeros~$\pm \alpha = \pm \pi / (1 + \varphi (u))$ closest
  to the origin, with
  \begin{equation}
    \varphi (u) = 2 (\cosh (\pi) - 1) u + \Omicron (u^2), \quad u \rightarrow
    0. \label{eq:bernoulli asympt sing}
  \end{equation}
  Furthermore, if $u \leqslant 2^{- 16}$, then one has $0 \leqslant \varphi
  (u) \leqslant 2 (\cosh (\pi) - 1) u$, and $h$~has no other zero than $\pm
  \alpha$ in the disk $\abs{w} < \rho \assign 6.2$.
\end{lemma}

\begin{proof}
  To start with, observe that when $u = 0$, the term $\tilde{S}^{\ast} (w)$ in
  the definition of~$h$ vanishes identically, leaving us with $h (w) =
  \check{S}^{\ast} (w)^{- 1} = w^{- 1} \sin w$. The zeros of~$1 /
  \check{S}^{\ast}$ closest to the origin are located at $w = \pm \pi$, the
  next closest, at $w = \pm 2 \pi$ and hence outside the disk $\abs{w} < \rho$.
  
  Let us focus on the zero at $w = \pi$. Since $h' (\pi) = 1 / \pi$ for $u =
  0$, the Implicit Function Theorem applies and shows that, locally, the zero
  varies analytically with~$u$. One obtains the asymptotic
  form~\eqref{eq:bernoulli asympt sing} by implicit differentiation.
  
  We turn to the bounds on $\varphi (u)$. The power series expansion
  of~$\tilde{S}^{\ast}$ with respect to~$w$ has nonnegative coefficients, showing that $h (\pi) =
  \tilde{S}^{\ast} (\pi) \geqslant 0$. Now, with $a = 1 + u$ as
  in~\eqref{eq:bernoulli St}, isolate the first nonzero term of that series and
  write
  \[ \tilde{S}^{\ast} (w) = \frac{a^6 - 1}{6} w^2 + \frac{G (a^2 w) - G
     (w)}{w}, \qquad G (w) = \sinh (w) - w - \frac{w^3}{6} . \]
  Let $K = 2 (\cosh (\pi) - 1)$ and $w_0 = \pi / (1 + Ku)$. One has
  \begin{equation} \label{eq:bernoulli diff G}
     G (a^2 w_0) - G (w_0) \leqslant (a^2 - 1) w_0 \max_{1 \leqslant t
     \leqslant a^2} \abs{G' (tw_0)} .
  \end{equation}
  Note that $a^2 w_0 \leqslant \pi$ for all~$u \leqslant 1 / 2$. Since $G (w) \gg 0$,
  it follows that
  \[ G (a^2 w_0) - G (w_0) \leqslant (a^2 - 1) w_0 G' (\pi), \]
  therefore the term $\tilde{S}^{\ast} (w_0)$ in $h (w_0)$ satisfies
  \[ \tilde{S}^{\ast} (w_0) \leqslant \frac{a^6 - 1}{6} w_0^2 + (a^2 - 1) G'
     (\pi) . \]
  As for the other term, the inequality $\sin (\pi - v) \geqslant v - v^3 / 6$
  ($u \geqslant 0$)
  applied to $v = \pi - w_0 = Kuw_0 $ yields
  \[ - \check{S}^{\ast} (w_0)^{- 1} \leqslant - Ku + (Ku)^3  \frac{w_0^2}{6} .
  \]
  Collecting both contributions and substituting in the value $G' (\pi) = (K -
  \pi^2) / 2$, we obtain
  \[ h (w_0) \leqslant \hat{h} := \frac{a^6 - 1}{6}  \frac{\pi^2}{(1 + Ku)^2} +
     \frac{a^2 - 1}{2}  (K - \pi^2) - Ku + (Ku)^3  \frac{w_0^2}{6} . \]
  The right-hand side is an explicit rational function of~$u$, satisfying
  \[ \hat{h} \sim (- 2 \pi^2  (K - 1) + K / 2) u^2 \]
  as $u \rightarrow 0$. One can check that it remains negative for $0 < u
  \leqslant 2^{- 16}$. Thus, one has $h (w_0) < 0 \leqslant h (\pi)$, and $h$
  has a zero in the interval $w \in [w_0, \pi]$, that is, a zero of the form $\alpha
  = \pi / (1 + \varphi (u))$ with $0 \leqslant \varphi (u) \leqslant Ku$, as
  claimed. The corresponding statement for $- \alpha$ follows by parity.
  
  It remains to show that $\pm \alpha$ are the only zeros of~$h$ in the
  disk~$\abs{w} < \rho$. We do it by comparing them to the zeros of~$1 /
  \check{S}^{\ast}$ using Rouché's theorem.
  
  To this end, note first that the expression $\abs{\sin (\rho e^{i \theta})}$
  reaches its minimum for $\theta \in [0, \pi / 2]$ when~$\theta = 0$. Indeed,
  one has $\abs{\sin (\rho e^{i \theta})}^2 = \cosh (\rho \sin \theta)^2 - \cos
  (\rho \cos \theta)^2$. The term $\cosh (\rho \sin \theta)^2$ is strictly
  increasing on the whole interval. For $\theta \leqslant \theta_0 \assign
  \arccos (3 \pi / (2 \rho))$, the term $- \cos (\rho \cos \theta)^2$ is
  increasing as well, so that $\abs{\sin (\rho e^{i \theta})} \geqslant \abs{\sin
  \rho}$ in that range, whereas for $\theta_0 \leqslant \theta \leqslant
  \pi$, we have $\abs{\sin (\rho e^{i \theta})}^2 \geqslant \cosh (\rho \sin
  \theta_0)^2 - 1 > 1 \geqslant \abs{\sin \rho}^2$. It follows that $\abs{\sin w}
  \geqslant \abs{\sin \rho}$, and hence $\abs{\check{S}^{\ast} (w)^{- 1}} \geqslant
  \rho^{- 1}  \abs{\sin \rho}$, for $w = \rho e^{i \theta}$ with $0 \leqslant
  \theta \leqslant \pi / 2$, and by symmetry on the whole circle~$\abs{w} =
  \rho$.
  
  Turning to $\tilde{S}^{\ast} (w)$
  and reasoning as in~\eqref{eq:bernoulli diff G}
  starting from~\eqref{eq:bernoulli St},
  one can write
  \[ \abs{\tilde{S}^{\ast} (w)} \leqslant (a^2 - 1) \sup_{1 \leqslant t
     \leqslant a^2} \abs{\cosh (tw) - 1} \leqslant (a^2 - 1) \cosh (a^2 \rho)
  \]
  for $\abs{w} = \rho$.
  Comparing these bounds leads to
  \[ \abs{\tilde{S}^{\ast} (w)} < 10^{- 2} < \abs{\check{S}^{\ast} (w)^{- 1}},
     \qquad \abs{w} = \rho, \quad 0 \leqslant u \leqslant 2^{- 16} . \]
  Therefore, $h (w)$ has the same number of zeros as $w^{- 1} \sin w$ inside
  the circle.
\end{proof}

\begin{proposition}
  \label{prop:bernoulli main}For all $0 < u \leqslant 2^{- 16}$, we have
  $\tilde{b}_k = b_k  (1 + \eta_k)$ where
  \[ \abs{\eta_k} \leqslant (1 + 21.2 u)^k  (1.1 k + 446) u. \]
\end{proposition}

\begin{proof}
  We use the notation of the previous lemma. In the
  expression~\eqref{eq:bernoulli ineq sol} of~$\hat{\delta} (z)$, the series
  $\tilde{S} (z)$ and $\tilde{C} (z)$ define entire functions, while
  $\minmaj{b}^{\ast} (w) = 2 - w / \tan w$ is a meromorphic function with
  poles at $w \in \pi \mathbb{Z}$. These observations combined with
  Lemma~\ref{lem:bernoulli dom root} imply that $\hat{\delta}^{\ast}$~is
  meromorphic in the disk $\abs{w} < 6.2$, with exactly four simple poles
  located at $w = \pm \alpha$ and $w = \pm \pi$.
  Only the first two poles depend on~$u$.
  
  One has $\hat{\delta}^{\ast} (w) \sim - (1 - w / \pi)^{- 1}$ as $w
  \rightarrow \pi$. Let $F^{\ast}$ denote the derivative of $w \mapsto w
  \tilde{S}^{\ast} (w)$. With the help of a computer algebra system, it is not
  too hard to determine that the singular expansion as $w \rightarrow \alpha$
  reads
  \[ \hat{\delta}^{\ast} (w) \sim \frac{\tilde{C}^{\ast} (\alpha^2) - \cos
     \alpha + 2 \alpha^{- 1} \sin \alpha}{F^{\ast}  (\alpha) - \cos \alpha} 
     \frac{1}{1 - w / \alpha} \backassign \frac{R (u)}{1 - w/\alpha} \]
  and one has $R (u) = 1 + (2 \cosh \pi - 2 - \pi \sinh \pi) u + \Omicron
  (u^2)$ as $u \rightarrow 0$. The expansions at $- \alpha$ and $- \pi$ follow
  since $\hat{\delta}^{\ast}$ is an even function. Set
  \[ \hat{\delta}^{\ast} (w) = \frac{R (u)}{1 - w / \alpha} + \frac{R (u)}{1 +
     w / \alpha} - \frac{1}{1 - w / \pi} - \frac{1}{1 + w / \pi} + g^{\ast}
     (u, w), \]
  where $g^{\ast} (u, \cdot)$ now is analytic for $\abs{w} < 6.2 \approx 1.97
  \pi$ and vanishes identically when $u = 0$. Since
  \[ \frac{R (u)}{1 \pm w / \alpha} - \frac{1}{1 \pm w / \pi} = \frac{R (u) - 1}{1
     \pm w / \alpha} \mp \frac{\varphi (u) w / \pi}{(1 \pm w / \alpha)  (1 \pm w /
     \pi)}, \]
  we have
  \[ \hat{\delta}^{\ast} (w) = \frac{2 (R (u) - 1)}{1 - (w / \alpha)^2} +
     \frac{2 \varphi (u)  (2 + \varphi (u))  (w / \pi)^2}{(1 - (w / \alpha)^2)
     (1 - (w / \pi)^2)} + g^{\ast} (u, w), \]
  that is,
  \[ \hat{\delta} (z) = \frac{2 (R (u) - 1)}{1 - z / (2 \alpha)^2} + \frac{2
     \varphi (u)  (2 + \varphi (u)) z / (2 \pi)^2}{(1 - z / (2 \alpha)^2)  (1
     - z / (2 \pi)^2)} + g (u, z) . \]
  By Cauchy's inequality, for any $\lambda < 1.9$, the Taylor coefficients~$g_k$
  of~$g (u, \cdot)$ satisfy
  \[ \abs{g_k} \leqslant \frac{A_{\lambda} (u)}{(2 \pi \lambda)^{2 k}}, \qquad
     A_{\lambda} (u) \assign \max_{\tmscript{\begin{array}{c}
       \abs{w} = \lambda \pi
     \end{array}}} \abs{g^{\ast} (u, w)}, \]
  and therefore
  \begin{equation}
    \hat{\delta} (z) \ll \frac{2 \abs{R (u) - 1}}{1 - z / (2 \alpha)^2} +
    \frac{2 \varphi (u)  (2 + \varphi (u)) z / (2 \alpha)^2}{(1 - z / (2
    \alpha)^2)^2} + \frac{A_{\lambda} (u)}{1 - z / (2 \pi \lambda)^2},
    \label{eq:bernoulli maj delta hat}
  \end{equation}
  where we have bounded $\nu z / (1 - \nu z)$ where $\nu = 1 / (2 \pi)^2$ by
  the same expression with $\nu = 1 / (2 \alpha)^2$.
  
  For $0 \leqslant u \leqslant 2^{- 16}$ and using the enclosure of~$\varphi
  (u)$ from Lemma~\ref{lem:bernoulli dom root}, a brute force evaluation using
  interval arithmetic yields the bounds
  \[ \abs{R (u) - 1} \leqslant (\max_{0 \leqslant v \leqslant u} \abs{R' (v)}) u
     \leqslant 72 u, \qquad \frac{2 \varphi (u)  (2 + \varphi (u))}{(2 \alpha)^2}
     \leqslant 2.2 u. \]
  SageMath code for computing these estimates can be found in the
  supplementary material. By the same method, choosing $\lambda = \sqrt{3 /
  2}$ and making use of the fact that $g^{\ast} (0, w) = 0$, we get
  \[ A_{\lambda} (u) \leqslant u \max_{\tmscript{\begin{array}{c}
       0 \leqslant v \leqslant u\\
       \abs{w} = \lambda \pi
     \end{array}}} \left| \frac{\partial g^{\ast}}{\partial u} (v, w) \right|
     \leqslant 747 u. \]
  We substitute these bounds in~\eqref{eq:bernoulli maj delta hat} to conclude
  that
  \begin{align*}
    \hat{\delta}_k & \leqslant (2 \times 72 (2 \alpha)^{- 2 k} + 2.2 k (2
    \alpha)^{- 2 k} + 747 (2 \pi \lambda)^{- 2 k}) u\\
    & \leqslant \frac{1}{(2 \pi)^{2 k}}  ((2.2 k + 2 \times 72)  (1 + 2
    (\cosh \pi - 1) u)^{2 k} + 747 (2 / 3)^k) u\\
    & \leqslant \frac{(1 + 21.2 u)^{2 k}}{(2 \pi)^{2 k}}  (2.2 k + 891) u.
  \end{align*}
  The claim follows since, as noted in Lemma~\ref{lem:bernoulli abs}, $\abs{b_k}
  \geqslant 2 (2 \pi)^{- 2 k}$ for all~$k \geqslant 1$.
\end{proof}

\begin{corollary} \label{cor:bernoulli main}
  For all $u$ and $k$ satisfying $0 < u \leqslant 2^{- 16}$ and $43 ku
  \leqslant 1$, one has $\tilde{b}_k = b_k  (1 + \eta_k)$ with $\abs{\eta_k}
  \leqslant (3 k + 1213) u$.
\end{corollary}

\begin{proof}
  The assumption on $ku$ implies $(1 + 21.2 u)^{2 k} \leqslant e$.
\end{proof}

\begin{figure}
  \centerline{\includegraphics{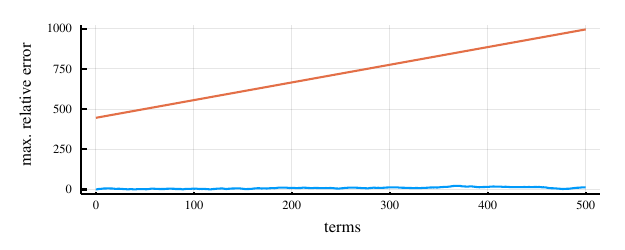}}
  \vspace*{-.5\baselineskip}
  \caption{%
    Bound from Proposition~\ref{prop:bernoulli main} (top curve) and
    measured error (bottom curve), in multiples of the unit roundoff~$u$,
    for the evaluation of the scaled Bernoulli numbers
    in 53-bit floating-point arithmetic with unbounded exponents.
    The script used to produce this plot is available in the supplementary material.
  }
  \label{fig:bernoulli}
\end{figure}

Proposition~\ref{prop:bernoulli main} and Corollary~\ref{cor:bernoulli main}
seem difficult to improve significantly, at least if we keep reasoning
analytically, without bringing into play the discrete low-level behavior of
rounding errors.
Figure~\ref{fig:bernoulli} illustrates that, as in the previous section, the bounds nevertheless overestimate the actual accumulated errors.

\begin{remark}
  \label{rk:bernoulli exp}Let us now prove our claim that no bound of the
  form~$\eta_k = \Omicron (ku)$ can hold uniformly with respect to
  $u$~and~$k$. In the notation of Lemma~\ref{lem:bernoulli local}, suppose
  that we have $s_k = u$ and $t_{k, j} = \hat{\theta}_{k - j}$ for all $k, j$.
  These values are reached, e.g., by taking $v_{k, i} = u$ and $v'_k = r_{k,
  j} = r'_k = 0$ for all~$i,j,k$. Since the $v_{k, i}$ and $v'_k$ represent
  individual rounding errors, this corresponds to a feasible situation in our
  model. Then, \eqref{eq:bernoulli err rec} translates into $\delta (z) S (z) =
  C (z) u - (\hat{\theta} \odot S) (z)  (b (z) + \delta (z))$, that is,
  \[ \delta (z) = \frac{C (z) u - \bigl(S ((1 + u) z) - S (z)\bigr) b (z)}{S ((1 + u)
     z)} . \]
  For small $u$, the denominator vanishes at $z = \beta \assign - 4 \pi^2 / (1
  + u)$, while the numerator is analytic for $\abs{w} < \pi$ and does not vanish
  at~$\beta$ (it tends to $-1$ as $u \rightarrow 0$, since $S(\beta) \sim u/2$ and $\delta(\beta) \sim -2 u^{-1}$). Thus,
  the radius of convergence of~$\delta (z)$ is at most $4 \pi^2 / (1 + u)$.
  This implies that $\eta_k$ grows exponentially for fixed~$u$.
\end{remark}

\section{Two Variables: The Equation of a Vibrating String}\label{sec:wave}

As part of an interesting case study of formal program verification
in scientific computing, Boldo {\tmem{et
al.}}~{\cite{Boldo2009,BoldoClementFilliatreMayeroMelquiondWeis2013,BoldoClementFilliatreMayeroMelquiondWeis2014}}
give a full worst-case rounding error analysis of a simple explicit finite
difference scheme for the one-dimensional wave equation
\begin{equation}
  \frac{\partial^2 p}{\partial t^2} - c^2  \frac{\partial^2 p}{\partial x^2} =
  0. \label{eq:wave}
\end{equation}
Such a finite difference scheme is nothing but a multi-dimensional linear
recurrence---in the present case, a two-dimensional one, with a time index
ranging over the natural numbers while the space index is restricted to a
finite domain. We rephrase the relatively subtle error analysis using a slight
extension of the language introduced in the previous sections. Doing so does
not change the essence of the argument, but possibly makes it more palatable.

We use the notation and assumptions of
{\cite[Section~3.2]{BoldoClementFilliatreMayeroMelquiondWeis2013}},
{\cite[Section~5]{BoldoClementFilliatreMayeroMelquiondWeis2014}}, except that
we flip the sign of $\delta_i^k$ to keep with our usual conventions. Time and
space are discretized into a grid with time step $\Delta t$ and space step
$\Delta x$. To the continuous solution~$p (x, t)$ corresponds a double
sequence $(p_i^k)$ where $k \geqslant 0$ is the space index and $0 \leqslant i
\leqslant n$ is the time index. Taking central differences for the derivatives
in~\eqref{eq:wave} and letting $a = (c \Delta t / \Delta x)^2$ leads to
\begin{equation}
  \left\{\begin{aligned}
    p_i^1 &= p_i^0 + \dfrac{a}{2}  (p_{i + 1} - 2 p_i + p_{i - 1}), & \\
    p_i^{k + 1} &= 2 p_i^k - p_i^{k - 1} + a (p_{i + 1}^k - 2 p_i^k + p_{i -
    1}^k), && k \geqslant 1.
  \end{aligned}\right. \label{eq:wave exact scheme}
\end{equation}
The problem is subject to the boundary conditions $p_0^k = p_n^k = 0$, $k \in
\mathbb{N}$, and we are given initial data~$(p_i^0)_{i = 1}^{n - 1}$. In
accordance with the Courant-Friedrichs-Lewy condition, we assume $0 < a
\leqslant 1$.

Boldo {\tmem{et al.}} study an implementation of~\eqref{eq:wave exact scheme} in
double-precision floating-point arithmetic, but their focus is on the propagation of absolute
errors. Their local error analysis shows that the computed
values~$(\tilde{p}_i^k)$ corresponding to~$(p_i^k)$ satisfy\footnote{The
correcting term involving $\delta_i^0$ in the expression of $\tilde{p}_i^1$
will help make the expression of the overall error more uniform.}
\begin{equation}
  \left\{\begin{aligned}
    \tilde{p}_i^0 &= p_i^0 + \delta_i^0\\
    \tilde{p}_i^1 &= \tilde{p}_i^0 + \dfrac{a}{2}  (\tilde{p}_{i + 1}^0 - 2
    \tilde{p}_i^0 + \tilde{p}_{i - 1}^0) + \delta_i^1 - \left( \delta_i^0 +
    \dfrac{a}{2}  (\delta_{i + 1}^0 - 2 \delta_i^0 + \delta_{i - 1}^0)
    \right),\\
    \tilde{p}_i^{k + 1} &= 2 \tilde{p}_i^k - \tilde{p}_i^{k - 1} + a
    (\tilde{p}_{i + 1}^k - 2 \tilde{p}_i^k + \tilde{p}_{i - 1}^k) +
    \delta_i^{k + 1}, \qquad k \geqslant 1
  \end{aligned}\right. \label{eq:wave approx scheme}
\end{equation}
with $\abs{\delta_i^k} \leqslant \bar{\delta} \assign 78 \cdot 2^{- 52}$ for all
$i$~and all~$k \geqslant 1$. While the discussion of error propagation
in~{\cite{BoldoClementFilliatreMayeroMelquiondWeis2013,BoldoClementFilliatreMayeroMelquiondWeis2014}}
assumes~$\delta_i^0 = 0$, the machine-checked proof actually allows for
initial errors $\abs{\delta_i^0} \leqslant \bar{\delta}^0 \assign 14 \cdot 2^{-
52}$ and gives the additional bound $\abs{\delta_i^1} \leqslant \bar{\delta}^1
\assign 81 \cdot 2^{- 53}$.
Both in~{\cite{BoldoClementFilliatreMayeroMelquiondWeis2013,BoldoClementFilliatreMayeroMelquiondWeis2014}}
and here, the computation of these local bounds is the only step of the analysis
that depends on the assumption that the computation is run in double precision.

From there, they prove~\cite[Theorem~5.2]{BoldoClementFilliatreMayeroMelquiondWeis2014}
that one has
$\abs{\tilde p_i^k - p_i^k} \leqslant \frac{1}{2}  \bar{\delta}  (k + 1)  (k + 2)$
for all $i$ and $k$.
(As usual, a naive analysis would lead to an exponential bound.)
The main aim of this section is to give a new proof of this result.

The iterations~\eqref{eq:wave exact scheme} and \eqref{eq:wave approx scheme} are,
a priori, valid for $0 < i < n$ only, but it is not hard to see that they can
be made to hold for all~$i \in \mathbb{Z}$ by extending the
sequences~$(p_i^k)$, $(\tilde{p}_i^k)$, and $(\delta_i^k)$ by odd symmetry and
$(2 n)$-periodicity with respect to~$i$. Viewing time as the main variable, we
encode space-periodic sequences of period $2 n$ by generating series of the
form
\begin{equation}\label{eq:bivgen}
f (x, t) = \sum_{k = 0}^{\infty} \sum_{i = - n}^{n - 1} f_i^k x^i t^k \in
   \Omega [[t]]
\end{equation}
where $\Omega =\mathbb{R} [x] / \langle x^{2 n} - 1 \rangle$ is the ring of
polynomials modulo $x^{2n} - 1$.
When $f$ is an element of $\Omega [[t]]$, we denote
\[ f^{u :} (x, t) = \sum_{k \geqslant u} \sum_i f_i^k x^i t^k, \qquad f_i (t)
   = \sum_k f_i^k t^k, \qquad f^k (x) = \sum_i f_i^k x^i . \]
Multiplication by $x$~and~$t$ in $\Omega [[t]]$ respectively corresponds to
backward shifts of the indices $i$~and~$k$.

Let $\Delta (x, t) = \tilde{p} (x, t) - p (x, t)$
(where $\Delta$, $\tilde p$, and~$p$ all are special cases of~\eqref{eq:bivgen})
be the generating series of
the global error. Our goal is to obtain bounds on the $\Delta_i^k$. We start
by expressing $\Delta (x, t)$ in terms of $\delta (x, t)$. This is effectively
a more precise version
of~{\cite[Theorem~5.1]{BoldoClementFilliatreMayeroMelquiondWeis2014}} covering
initial data with numeric errors.

\begin{proposition}
  \label{prop:wave expr}One has
  \begin{equation}
    \Delta (x, t) = \lambda (x, t) \eta (x, t) \label{eq:wave gen series}
  \end{equation}
  where
  \begin{equation}
    \lambda (x, t) = \frac{1}{1 - \varphi (x) t + t^2}, \qquad \eta (x, t) =
    \delta (x, t) - \varphi (x) t \delta_0 (x) \label{eq:wave lambda eta}
  \end{equation}
  with
  \[ \varphi (x) = 2 + a (x^{- 1} - 2 + x) . \]
\end{proposition}

\begin{proof}
  By comparing \eqref{eq:wave approx scheme} with \eqref{eq:wave exact scheme} and
  observing how~$\Delta_i^1$ simplifies thanks to the correcting term
  in~$\tilde{p}_i^1$, we get
  \begin{equation}
    \left\{\begin{aligned}
      \Delta_i^k &= \delta_i^k, && k = 0, 1,\\
      \Delta_i^{k + 1} &= 2 \Delta_i^k - \Delta_i^{k - 1} + a (\Delta_{i + 1}^k
      - 2 \Delta_i^k + \Delta_{i - 1}^k) + \delta_i^{k + 1}, && k + 1 \geqslant
      2.
    \end{aligned}\right. \label{eq:num scm}
  \end{equation}
  In terms of series, \eqref{eq:num scm} translates into
  \begin{gather}
    \Delta^0 (x) = \delta^0 (x), \qquad \Delta^1 (x) = \delta^1 (x),
    \label{eq:wave series init} \\
    t^{- 1} \Delta^{2 :} (x, t) = 2 \Delta^{1 :} (x, t) - t \Delta (x, t) + a
    (x^{- 1} - 2 + x) \Delta^{1 :} (x, t) + t^{- 1} \delta^{2 :} (x, t) .
    \label{eq:wave series main}
  \end{gather}
  Taking into account~\eqref{eq:wave series init}, equation~\eqref{eq:wave series
  main} becomes
  \[ \Delta (x, t) = 2 t (\Delta (x, t) - \delta_0 (x)) - t^2 \Delta (x, t) +
     a (x^{- 1} - 2 + x) t (\Delta (x, t) - \delta^0 (x)) + \delta (x, t), \]
  that is,
  \[ (1 - (2 + a (x^{- 1} - 2 + x)) t + t^2) \Delta (x, t) = \delta (x, t) -
     \varphi (x) t \delta_0 (x) \]
  The coefficient on the left-hand side is invertible in~$\Omega [[t]]$,
  leading to the equality $\Delta (x, t) = {\lambda (x, t)} \eta (x, t)$.
\end{proof}

Rather than by bivariate majorant series, we will control elements of $\Omega
[[t]]$ by majorant series in a single variable relative to a norm on~$\Omega$.
We say that $\hat{f} \in \mathbb{R}_{\geqslant 0} [[t]]$ is a majorant series
of $f \in A [[t]]$ with respect to a norm $\|{\cdot}\|_s$ on an algebra~$A$,
and write $f \ll_s \hat{f}$, when $\| f_k \|_s \leqslant \hat{f}_k$ for all $k
\in \mathbb{N}$. The basic properties listed in Section~\ref{sec:maj} extend
in the obvious way. In particular, if $\|{\cdot}\|_q, \|{\cdot}\|_r, \| {\cdot}
\|_s$ are norms such that $\| uv \|_q \leqslant \| u \|_r  \| v \|_s$, then $f
\ll_r \hat{f}$ and $g \ll_s \hat{g}$ imply $fg \ll_q \hat{f}  \hat{g}$.

For $u = \sum_i u_i x^i \in \Omega$, we define
\[ \| u \|_1 = \sum_{i = - n}^{n - 1} \abs{u_i}, \qquad \| u \|_2 = \left(
   \sum_{i = - n}^{n - 1} \abs{u_i}^2 \right)^{1 / 2}, \qquad \| u \|_{\infty} =
   \max_{i = - n}^{n - 1}  \abs{u_i} . \]
Note the inequality $\| uv \|_{\infty} \leqslant \| u \|_1  \| v \|_{\infty}$
for $u, v \in \Omega$ (an instance of Young's convolution inequality).

The local error analysis yields
\[ \delta (x, t) \ll_{\infty} \bar{\delta}^0 + \bar{\delta}^1 t +
   \frac{\bar{\delta} t^2}{1 - t} . \]
Using the notation of Proposition~\ref{prop:wave expr}, one has $\| \varphi
(x) \|_1 \leqslant 2$, hence $\| \varphi (x) \delta^0 (x) \|_{\infty}
\leqslant 2 \bar{\delta}^0$ and
\begin{equation}
  \eta (x, t) \ll_{\infty} \bar{\delta}^0 + (\bar{\delta}^1 + 2
  \bar{\delta}_0) t + \frac{\bar{\delta} t^2}{1 - t} \ll \frac{\bar{\delta}}{1
  - t} . \label{eq:wave bound eta}
\end{equation}
We first deduce a bound on the global error in quadratic mean with respect to
space. We will later get a second proof of this result as a corollary of
Proposition~\ref{prop:wave uniform norm}; the main interest of the present one
is that it does not rely on Lemma~\ref{lem:wave positivity}.

\begin{proposition}
  \label{prop:wave two norm}One has
  \[ \Delta (x, t) \ll_2 \frac{\sqrt{2 n}}{(1 - t)^3}  \bar{\delta} . \]
  In other words, the root mean square error at time~$k$ satisfies
  \[ \left( \frac{1}{n}  \sum_{i = 0}^{n - 1} (\Delta_i^k)^2 \right)^{1 / 2}
     \leqslant \frac{(k + 1)  (k + 2)}{2}  \bar{\delta} . \]
\end{proposition}

\begin{proof}
  Elements of $\Omega$ can be evaluated at $2 n$-th roots of unity, and the
  collection $u^{\ast} = (u (\omega))_{\omega^{2 n} = 1}$ of values of a
  polynomial~$u \in \Omega$ is nothing but the discrete Fourier transform of
  its coefficients. The coefficientwise Fourier transform of formal power
  series,
  \[ f (t) \mapsto f^{\ast} (t) = (f (\omega, t))_{\omega^{2 n} = 1}, \]
  is an algebra homomorphism from $\Omega [[t]]$ to $\mathbb{C}^{2 n} [[t]]$.
  One easily checks Parseval's identity for the discrete Fourier transform:
  \[ \| u \|_2 = \frac{1}{\sqrt{2 n}}  \| u^{\ast} \|_2, \qquad u \in
     \Omega, \]
  where the norm on the right-hand side is the standard Euclidean norm on
  $\mathbb{C}^{2 n}$.
  
  The uniform bound~\eqref{eq:wave bound eta} resulting from the local error
  analysis implies
  \[ \eta (x, t) \ll_2 \frac{\sqrt{2 n}}{1 - t}  \bar{\delta}, \]
  and Parseval's identity yields
  \[ \eta^{\ast} (t) \ll_2 \frac{2 n \bar{\delta}}{1 - t} . \]
  At $x = \omega = e^{\mathi \theta}$, the factor $\lambda$ in~\eqref{eq:wave
  gen series} takes the form
  \[ \lambda (\omega, t) = \frac{1}{1 - 2 b_{\omega} t + t^2}, \qquad
     b_{\omega} = 1 + a (\cos \theta - 1), \]
  where $- 1 \leqslant b_{\omega} \leqslant 1$ due to the assumption that $0 <
  a \leqslant 1$. The denominator therefore factors as $(1 - \zeta_{\omega} t)
  (1 - \bar{\zeta}_{\omega} t)$ where $\abs{\zeta_{\omega}} = 1$, so that we
  have
  \[ \lambda (\omega, t) = \frac{1}{(1 - \zeta_{\omega} t) (1 -
     \bar{\zeta}_{\omega} t)} \ll \frac{1}{(1 - t)^2} \]
  in $\mathbb{C} [[t]]$ and hence $\lambda^{\ast} (t) \ll_{\infty} (1 - t)^{-
  2}$.
  
  Since the entrywise product in~$\mathbb{C}^{2 n}$ satisfies $\| u^{\ast}
  v^{\ast} \|_2 \leqslant \| u^{\ast} \|_{\infty}  \| v^{\ast} \|_2$, the
  bounds on $\lambda^{\ast}$ and $\eta^{\ast}$ combine into
  \[ \lambda^{\ast} (t) \eta^{\ast} (t) \ll_2 \frac{2 n \bar{\delta}}{(1 -
     t)^3} . \]
  Using Parseval's identity again, we conclude that
  \[ \Delta (x, t) = \lambda (x, t) \eta (x, t) \ll_2 \frac{\sqrt{2 n} 
     \bar{\delta}}{(1 - t)^3} \]
  as claimed. The second formulation of the result comes from the symmetry of
  the data: one has $\Delta (x^{- 1}, t) = - \Delta (x, t)$, hence $\|
  \Delta^k \|_2^2 = 2 \sum_{i = 0}^{n - 1} (\Delta_i^k)^2$ for all~$k$.
\end{proof}

Proposition~\ref{prop:wave two norm} immediately implies $\Delta (x, t)
\ll_{\infty} \bar{\delta}  \sqrt{2 n}  (1 - t)^{- 3}$. However, this
estimate turns out to be too pessimistic by a factor $\sqrt{2 n}$. The key
to better bounds is the following lemma, proved (with the help of generating
series!) in Appendix~C of
{\cite{BoldoClementFilliatreMayeroMelquiondWeis2014}}. The argument, due to
M.~Kauers and V.~Pillwein, reduces the problem to an inequality of Askey and
Gasper via an explicit expression in terms of Jacobi polynomials that is
proved using Zeilberger's algorithm. It remains intriguing to find a more
direct way to derive the uniform bound on~$\Delta$.

\begin{lemma}
  \label{lem:wave positivity}The coefficients $\lambda_i^k$ of $\lambda (x,
  t)$ are nonnegative.
\end{lemma}

Strictly speaking, the nonnegativity result
in~{\cite{BoldoClementFilliatreMayeroMelquiondWeis2014}} is about the
coefficients, not of $\lambda (x, t) \in \Omega [[t]]$ as defined above, but
of its lift to $\mathbb{R} [x, x^{- 1}] [[t]]$ obtained by
interpreting~\eqref{eq:wave lambda eta} in the latter ring. It is thus slightly
stronger than the above lemma.

From this lemma it is easy to deduce a more satisfactory bound on the global
error, matching that
of~{\cite[Theorem~5.2]{BoldoClementFilliatreMayeroMelquiondWeis2014}}.

\begin{proposition}
  \label{prop:wave uniform norm}One has the bound
  \[ \Delta (x, t) \ll_{\infty} \frac{\bar{\delta}}{(1 - t)^3}, \]
  that is,
  \[ \abs{\Delta_i^k} \leqslant \frac{1}{2}  \bar{\delta}  (k + 1)  (k + 2)  \]
  for all $i$ and $k$.
\end{proposition}

\begin{proof}
  Lemma~\ref{lem:wave positivity} implies $\lambda (x, t) \ll_1 \lambda (1,
  t)$. This bound combines with \eqref{eq:wave bound eta} to give
  \[ \Delta (x, t) = \lambda (x, t) \eta (x, t) \ll_{\infty} \lambda (1, t) 
     \frac{\bar{\delta}}{(1 - t)} = \frac{\bar{\delta}}{(1 - t)^3}, \]
  just like in the proof of Proposition~\ref{prop:wave two norm}.
\end{proof}

\section{Solutions of Linear Differential Equations}\label{sec:dfinite}

For the last application, we return to recurrences of finite order in a single
variable. Instead of looking at a specific sequence, though, we now consider a
general class of recurrences with polynomial coefficients. It is technically
simpler and quite natural to restrict our attention to recurrences associated
to nonsingular differential equations under the correspondence from
Lemma~\ref{lem:recdeq}: thus, we consider a linear ordinary differential
equation
\begin{equation}
  p_r (z) y^{(r)} (z) + \cdots + p_1 (z) y' (z) + p_0 (z) y (z) = 0
  \label{eq:deq}
\end{equation}
where $p_0, \ldots, p_r \in \mathbb{C} [z]$ and assume that $p_r (0) \neq 0$.
We expect that this assumption could be lifted by working along the lines
of~{\cite{MezzarobbaSalvy2010}}.

It is classical that~\eqref{eq:deq} then has~$r$ linearly independent formal
power series solutions and all these series are convergent in a neighborhood
of~$0$. Suppose that we want to evaluate one of these solutions at a point
lying within its disk of convergence. A natural way to proceed is to sum the
series iteratively, using the associated recurrence to generate the
coefficients. Our goal in this section is to give an error bound on the
approximation of the partial sum computed by a version of this
algorithm.
We do not consider the {\tmem{truncation}} error here
(see however Remark~\ref{rk:truncations} below).

We formulate the computation as an algorithm based on interval arithmetic that
returns an enclosure of the partial sum. Running the whole loop in interval
arithmetic would typically lead to enclosures of width that growths
exponentially with the number of computed terms, and thus to a catastrophic
loss of accuracy. Instead, the algorithm executes the body of the loop in
interval arithmetic, which saves us from going into the details of the local
error analysis, but ``squashes'' the computed interval to its midpoint after
each loop iteration. It maintains a running bound on the discarded radii that
serves to control the overall effect of the propagation of local errors. This
way of using interval arithmetic does not create long chains of interval
operations depending on each other and only produces a small overestimation.

The procedure is presented as Algorithm~\ref{algo:dfsum}.
While, to the best of our knowledge, the algorithm is new,
the approach just sketched is a very natural one.
The main contribution of this section is the automated error analysis that makes it applicable.

In the algorithm and the analysis that follows, we use the notation of
Section~\ref{sec:genseries} for differential and recurrence operators, with
the symbol~$\partial$ denoting~$\mathd / \mathd z$. When $P = p_r \partial^r +
\cdots$ is a differential operator, we denote by
\[ \rho (P) = \min \{ \abs{\xi} \of p_r (\xi) = 0 \} \in [0, \infty] \]
the radius of the disk centered at the
origin and extending to the nearest singular point. Variable names set in bold
represent complex intervals, or {\tmem{balls}}, and operations involving them
obey the usual laws of midpoint-radius interval arithmetic (i.e.,
$\tmmathbf{x} \ast \tmmathbf{y}$ is a reasonably tight ball containing $x \ast
y$, for all $x \in \tmmathbf{x}$, $y \in \tmmathbf{y}$, and for every
arithmetic operation~$\ast$). We denote by $\tmop{mid} (\tmmathbf{x})$ the
center of a ball~$\tmmathbf{x}$ and by $\tmop{rad} (\tmmathbf{x})$ its radius.

\begin{algorithm}[p!]
\caption{}
\label{algo:dfsum}
\begin{minipage}{\textwidth}
\begin{description}
  \item[Input] An operator $P = p_r (z) \partial^r + \cdots + p_1 (z) \partial
  + p_0 (z)$, with $p_r (0) \neq 0$. A vector $(\tmmathbf{u}_n)_{n = 0}^{r -
  1}$ of ball initial values. An evaluation point $\zeta \in \mathbb{C}$ with
  $\abs{\zeta} < \rho (P)$. A truncation order~$N$.
  
  \item[Output] A complex ball containing $\sum_{n = 0}^{N - 1} u_n \zeta^n$,
  where $u (z)$ is the solution of $P \cdot u = 0$ corresponding to the given
  initial values.
\end{description}
\begin{enumeratenumeric}
  \item \label{step:recop}[Compute a recurrence relation.] Define $L \in
  \mathbb{K} [X] [Y]$ by $z^r P = L (z, z \partial)$. Compute polynomials $b_0
  (n), \ldots, b_s (n)$ such that
  \[ L (S^{- 1}, n) = b_0 (n) - b_1 (n) S^{- 1} - \cdots - b_s (n) S^{- s} .
  \]
  \item \label{step:majeq}[Compute a majorant differential equation.] Let $m =
  \max (1, \deg p_r)$. Compute a rational~$\alpha \approx \rho (P)^{- 1}$ such
  that $\rho (P)^{- 1} < \alpha < \abs{\zeta}^{- 1}$. (If $\rho (P) = \infty$,
  take, for instance, $\alpha \approx \abs{\zeta}^{- 1} / 2$.) Compute rationals
  $c \approx \abs{p_r (0)}$ and $M$ such that $0 < c \leqslant \abs{p_r (0)}$ and
  $M \geqslant \max_{i = 0}^{r - 1} \sum_{j = 0}^{\deg p_i} \abs{p_{i, j}}
  \alpha^{- i - j - 1}$.
  
  \item \label{step:maj ini}[Initial values for the bounds.] Compute positive
  lower bounds on the first~$r$ terms of the series $\hat{g} (z) = \exp \left(
  Mc^{- 1} \alpha \int_0^z (1 - \alpha z)^{- m} \right)$. (This is easily done
  using arithmetic on truncated power series with ball coefficients.) Deduce
  rationals $\hat{u}_0 \geqslant \max_{n = 0}^{r - 1} (\abs{\tmmathbf{u}_n} /
  \hat{g}_n)$ and $\hat{\delta}_0 \geqslant \max_{n = 0}^{r - 1} (\abs{\delta_n}
  / \hat{g}_n)$.
  
  \item \relax[Initialize the recurrence.] Set \[(\tilde{u}_{r - s}, \ldots, \tilde{u}_{- 1},
  \tilde{u}_0, \ldots, \tilde{u}_{r - 1}) = (0, \ldots, 0, \tmop{mid}
  (\tmmathbf{u}_0), \ldots, \tmop{mid} (\tmmathbf{u}_{r - 1})), \]
  $\tmmathbf{s}_0 = 0$, $\tmmathbf{t}_0 = 1$, $\bar{\eta} = 0$. (Although most
  variables are indexed by functions on~$n$ for ease of reference, only
  $(\tilde{u}_{n - i})_{i = 0}^s$, $\tmmathbf{s}_n$, $\tmmathbf{t}_n$, and
  $\bar{\eta}$ need to be stored from one loop iteration to the next.)
  
  \item \label{step:for}For $n = 1, \ldots, N$, do:
  \begin{enumerate}
    \item \label{step:if}If $n \geqslant r$, then:
    \begin{enumeratenumeric}
      \item \label{step:rec}[Next coefficient.] Compute
      \[ \tmmathbf{u}_n = \dfrac{1}{b_0 (n)}  (b_1 (n)  \tilde{u}_{n - 1} +
         \cdots + b_s (n)  \tilde{u}_{n - s}) \]
      in ball arithmetic.
      
      \item \relax[Round.] \label{step:round}
      Set $\tilde{u}_n = \tmop{mid} (\tmmathbf{u}_n)$.
      (If $\tmmathbf{u}_n$ contains~$0$, it can be
       better in practice to force $\tilde{u}_n$ to~$0$ even if $\tmop{mid}
       (\tmmathbf{u}_n) \neq 0$ and increase~$\tmop{rad} (\tmmathbf{u}_n)$
       accordingly.)
      
      \item \relax[Local error bound.] \label{step:local}Let $\mu = \abs{\tilde{u}_{n -
      1}} + \cdots + \abs{\tilde{u}_{n - s}}$. If $\mu \neq 0$, then
      update~$\bar{\eta}$ to $\max (\bar{\eta}, \eta_n)$ where $\eta_n =
      \tmop{rad} (\tmmathbf{u}_n) / \mu$.
    \end{enumeratenumeric}
    \item \relax[Next partial sum.] Compute $\tmmathbf{t}_n = \zeta \cdot
    \tmmathbf{t}_{n - 1}$ and $\tmmathbf{s}_n =\tmmathbf{s}_{n - 1} +
    \tilde{u}_{n - 1} \tmmathbf{t}_{n - 1}$ using ball arithmetic.
  \end{enumerate}
  \item \label{step:add-error}[Account for accumulated numerical errors.]
  Compute $\sigma \geqslant c (\abs{\zeta} + \cdots + \abs{\zeta}^s)$. If $\sigma
  \bar{\eta} \geqslant 1$, signal an error. Otherwise, compute
  \begin{equation}
    A \geqslant \frac{M \alpha \abs{\zeta}}{(1 - \alpha \abs{\zeta})^m}, \qquad
    \Delta_N \geqslant \frac{\hat{\delta}_0 + \hat{u}_0 \sigma (1 + A) 
    \bar{\eta}}{1 - \sigma \bar{\eta}} \exp \frac{A}{1 - \sigma \bar{\eta}}
    \label{eq:df final rad}
  \end{equation}
  and increase the radius of~$\tmmathbf{s}_N$ by~$\Delta_N$.
  
  \item Return $\tmmathbf{s}_N$.
\end{enumeratenumeric}
\end{minipage}
\end{algorithm}

As mentioned, the key feature of Algorithm~\ref{algo:dfsum} is that
step~\ref{step:rec} computes~$\tmmathbf{u}_n$
based only on the centers of the intervals $\tmmathbf{u}_{n - 1}, \ldots,
\tmmathbf{u}_{n - s}$, ignoring their radii. In the remainder of this section, we will prove that thanks to the
correction made at step~\ref{step:add-error}, the enclosure returned when the
computation succeeds is nevertheless correct. The algorithm may also fail at
step~\ref{step:add-error}, but that can always be avoided by increasing the
working precision (provided that interval operations on inputs of radius
tending to zero produce results of radius that tends to zero).

With the notation from the algorithm, let $u (z)$ be a power series solution
of $P \cdot u = 0$ corresponding to initial values $u_0 \in \tmmathbf{u}_0,
\ldots, u_r \in \tmmathbf{u}_r$. The Cauchy existence theorem implies that
such a solution exists and that $u (z)$ converges on the disk $\abs{z} < \rho
(P)$. We recall basic facts about the recurrence obtained at
step~\ref{step:recop}. Let $Q (X) = X (X - 1) \cdots (X - r + 1)$.

\begin{lemma}
  \label{lem:df indpol}The coefficient sequence~$(u_n)$ of $u (z)$ satisfies
  \begin{equation}
    b_0 (n) u_n = b_1 (n) u_{n - 1} + \cdots + b_s (n) u_{n - s}, \label{eq:df
    exact rec} \nopagebreak
  \end{equation}
  where one has $b_0 (n) = p_r (0) Q (n)$. In particular, $b_0$ is not the
  zero polynomial.
\end{lemma}

\begin{proof}
  Observe that for $p \in \mathbb{C} [z]$, the operator $\partial p = p
  \partial + p'$ has~$p$ as a leading coefficient when viewed as a polynomial
  in~$\partial$ with coefficients in~$\mathbb{C} [z]$ written to the left. It
  follows that $z^k \partial^k = (z \partial)^k + \left( \text{terms involving
  lower powers of~$\partial$} \right)$ and therefore that $z^r P$ can be
  written as a polynomial in $z$ and $z \partial$, as implicitly required by
  the algorithm. That the operator $L (S^{- 1}, n)$ annihilates~$(u_n)$
  follows from Lemma~\ref{lem:recdeq}.
  
  The only term of $z^r P$, viewed as a sum of monomials $p_{i, j} z^j
  \partial^i$ that can contribute to $b_0$ is $p_r (0) z^r \partial^r$, for
  all others have $i - j < 0$. The relation $z^k \partial^k = (z \partial - {k
  + 1}) z^{k - 1} \partial^{k - 1}$ then shows that $b_0 (n) = p_r (0) Q (n)$,
  where $p_0 (0) \neq 0$ by assumption.
\end{proof}

Let $\hat{a} (z) = Mc^{- 1} \alpha (1 - \alpha z)^{- m}$ where $M$, $c$, and
$\alpha$ are the quantities computed at step~\ref{step:majeq} of the
algorithm.

\begin{lemma}
  \label{lem:df maj eq}Let $y, \hat{y}$ be power series such that $P \cdot y
  \ll \partial^{r - 1}  (\partial - \hat{a}) \cdot \hat{y}$. If one has $\abs{y_n
 }
  \leqslant \hat{y}_n$ for $n < r$, then $y \ll \hat{y}$.
\end{lemma}

\begin{proof}
  Let $\hat{P} = \partial^{r - 1}  (\partial - \hat{a} (z))$, that is,
  \[ \hat{P} = \partial^{r - 1}  \left( \partial - \frac{Mc^{- 1} \alpha}{(1 -
     \alpha z)^m} \right) = \partial^r - \sum_{i = 0}^{r - 1} \hat{f}_i (z)
     \partial^{r - 1 - i} \]
  where
  \[ \hat{f}_i (z) = \binom{r - 1}{i}  \frac{(m + i - 1) !}{(m - 1) !} 
     \frac{Mc^{- 1} \alpha^{i + 1}}{(1 - \alpha z)^{m + i}} . \]
  The parameters $c$, $m$, and $\alpha$ are chosen so that $(1 - \xi^{- 1}
  z)^{- 1} \ll (1 - \alpha z)^{- 1}$ for every root~$\xi$ of~$p_r$, and hence
  $p_r (z)^{- 1} \ll c^{- 1}  (1 - \alpha z)^{- m}$. Since, additionally, one
  has $(\alpha z)^j  (1 - \alpha z)^{- 1} \ll (1 - \alpha z)^{- 1}$, it
  follows that, for $0 \leqslant i < r$,
  \[ \frac{p_i (z)}{p_r (z)} \ll \frac{\sum_j \abs{p_{i, j}} z^j}{c (1 - \alpha
     z)^m} = \sum_j \abs{p_{i, j}} \alpha^{- j}  \frac{(\alpha z)^j}{c (1 -
     \alpha z)^m} \ll \frac{Mc^{- 1} \alpha^{i + 1}}{(1 - \alpha z)^m} \ll
     \hat{f}_i (z), \]
  by definition of~$M$. By Proposition~\ref{prop:maj deq}, these inequalities
  and our assumptions on~$\hat{y}$ imply that one has ${y \ll \hat{y}}$.
\end{proof}

Lemma~\ref{lem:df maj eq} applies in particular to series $y, \hat{y}$ with $P
\cdot y = 0$ and $\hat{y}' = \hat{a}  \hat{y}$. The solution~$\hat{g}$ of the
latter equation with $\hat{g}_0 = 1$ is the series
\begin{equation}
  \hat{g} (z) = \exp \int_0^z \hat{a} (w) \mathd w \label{eq:df ghat}
\end{equation}
already encountered at step~\ref{step:maj ini} of the algorithm. Observe that
none of its coefficients vanishes. Therefore, step~\ref{step:maj ini} runs
without error, and ensures that $\abs{u_n} \leqslant \hat{u}_0  \hat{g}_n$ (and
$\abs{\delta_n} \leqslant \hat{\delta}_0  \hat{g}_n$) for $n < r$. As $P \cdot u
= 0$, the lemma implies $u \ll \hat{u}_0  \hat{g}$.

\begin{remark} \label{rk:truncations}
Therefore, the tails of the series
$\hat{u} (z) = \hat{u}_0  \hat{g} (z)$ determined by
Algorithm~\ref{algo:dfsum} are majorant series of the tails of~$u (z)$. This
means that the algorithm can be modified to simultaneously bound the
truncation and rounding error, and most of the steps involved can be shared between
both bounds.
We refer the reader to~{\cite{Mezzarobba2019}} and the references therein for more on the computation of tight bounds on truncation errors.
Though our focus here is on the propagation of local errors,
the modified algorithm is the more interesting one for applications in rigorous computing, for it can serve as the basic brick of an algorithm for computing rigorous enclosures of solutions of ODEs with polynomial coefficients.
\end{remark}

Let us now turn to the loop. As usual, consider the computed coefficient
sequence~$(\tilde{u}_n)$, and let $\delta_n = \tilde{u}_n - u_n$. Write
\[ \tilde{u}_n = \dfrac{1}{b_0 (n)}  (b_1 (n)  \tilde{u}_{n - 1} + \cdots +
   b_s (n)  \tilde{u}_{n - s}) + \varepsilon_n, \qquad r \leqslant n \leqslant
   N, \]
so that $\abs{\varepsilon_n} \leqslant \tmop{rad} (\tmmathbf{u}_n)$. Let $\eta_n
= \varepsilon_n / (\abs{\tilde{u}_{n - 1}} + \cdots + \abs{\tilde{u}_{n - s}})$
when $r \leqslant n \leqslant N$ and the denominator is nonzero, and $\eta_n =
0$ otherwise. We thus have, for all $n \in \mathbb{Z}$,
\begin{equation}
  b_0 (n)  \tilde{u}_n - b_1 (n)  \tilde{u}_{n - 1} - \cdots - b_s (n) 
  \tilde{u}_{n - s} = b_0 (n) \eta_n  (\abs{\tilde{u}_{n - 1}} + \cdots + |
  \tilde{u}_{n - s} |) \label{eq:df approx rec}
\end{equation}
and, thanks to step~\ref{step:local} of the
algorithm, $\abs{\eta_n} \leqslant \bar{\eta}$. By subtracting \eqref{eq:df exact
rec} from \eqref{eq:df approx rec} and bounding~$\eta_n$ by $\bar{\eta}$, we
obtain
\begin{equation}
  \abs{b_0 (n) \delta_n - b_1 (n) \delta_{n - 1} - \cdots - b_s (n) \delta_{n -
  s}} \leqslant c \bar{\eta} Q (n)  (\abs{\tilde{u}_{n - 1}} + \cdots +
  \abs{\tilde{u}_{n - s}}) . \label{eq:df approx rec ineq}
\end{equation}
Let $\varphi (z) = z + \cdots + z^s$.

\begin{lemma}
  \label{lem:df maj ineq}Let $\hat{v} \in \mathbb{R}_{\geqslant 0} [[z]]$ be
  any majorant series of $( \varphi \, \minmaj{u} )' (z)$. The
  equation
  \begin{equation}
    (1 - c \bar{\eta} \varphi (z))  \hat{\delta}' (z) = (\hat{a} (z) + c
    \bar{\eta} \varphi' (z))  \hat{\delta} (z) + c \bar{\eta}  \hat{v} (z) .
    \label{eq:df maj deq}
  \end{equation}
  admits a solution $\hat{\delta} (z)$ with the initial value $\hat{\delta}_0$
  computed at step~\ref{step:maj ini}, and this solution is a majorant series
  of~$\delta (z)$.
\end{lemma}

\begin{proof}
  In terms of generating series, \eqref{eq:df approx rec ineq} rewrites as $z^r
  P \cdot \delta (z) \ll c \bar{\eta} Q (z \partial) \varphi (z) \cdot
  \minmaj{\tilde{u} (z)}$. As already observed in the proof of
  Lemma~\ref{lem:df indpol}, one has $Q (z \partial) = z^r \partial^r$, so
  that the previous equation is equivalent to
  \begin{equation}
    P \cdot \delta (z) \ll c \bar{\eta} \partial^r \varphi (z) \cdot
    \minmaj{\tilde{u} (z)} . \label{eq:df diff ineq}
  \end{equation}
  Let $\hat{\gamma}$ be the solution of
  \begin{equation}
    (\partial - \hat{a} (z)) \cdot \hat{\gamma} (z) = c \bar{\eta} \partial
    \varphi (z) \cdot \minmaj{\tilde{u} (z)} \label{eq:df maj eq1}
  \end{equation}
  with $\hat{\gamma}_0 = \hat{\delta}_0$. By Proposition~\ref{prop:maj deq},
  we have $\hat{\delta}_0  \hat{g} (z) \ll \hat{\gamma} (z)$ where
  $\hat{g}$~is given by~\eqref{eq:df ghat}. In addition, as noted when
  discussing step~\ref{step:maj ini}, we have $\abs{\delta_n} \leqslant
  \hat\delta_0  \hat{g}_n$ for $n < r$, hence $\abs{\delta_n} \leqslant
  \hat{\gamma}_n$ for $n < r$. As $\hat{\gamma}$ also satisfies $P \cdot
  \delta \ll \partial^r  (\partial - \hat{a}) \cdot \hat{\gamma}$, we can
  conclude that $\delta \ll \hat{\gamma}$ using Lemma~\ref{lem:df maj eq}. But
  then, since $\tilde{u} = u + \delta$, we have $\minmaj{\tilde{u} \ll
  \minmaj{u} + \hat{\gamma}}$, hence $\left( \varphi \minmaj{\tilde{u}}
  \right)' \ll \hat{v} + (\varphi \hat{\gamma})'$, and \eqref{eq:df maj
  eq1}~implies
  \[ \hat{\gamma}' = \hat{a}  \hat{\gamma} + c \bar{\eta}  \left( \varphi
     \minmaj{\tilde{u}} \right)' \ll c \bar{\eta} \varphi \hat{\gamma}' +
     (\hat{a} + c \bar{\eta} \varphi')  \hat{\gamma} + c \bar{\eta}  \hat{v}
  \]
  where we note that $\varphi (0) = 0$. This inequality is of the form
  required by Lemma~\ref{lem:maj diff ineq}, which yields the existence of
  $\hat{\delta}$ and the inequality $\hat{\gamma} \ll \hat{\delta}$. We thus
  have $\delta \ll {\hat{\gamma} \ll \hat{\delta}}$.
\end{proof}

It remains to solve the majorant equation~\eqref{eq:df maj deq} to get an
explicit bound on~$\hat{\delta}$.

\begin{proposition}
  \label{prop:df maj series}The generating series~$\delta (z)$ of the global
  error on $u_n$ committed by Algorithm~\ref{algo:dfsum} satisfies
  \begin{equation}
    \delta (z) \ll \frac{\hat{\delta}_0 + c \bar{\eta}  \hat{u}_0 \varphi (z) 
    (1 + z \hat{a} (z))}{1 - c \bar{\eta} \varphi (z)} \exp \left( \frac{z
    \hat{a} (z)}{1 - c \bar{\eta} \varphi (z)}  \right), \quad \hat{a} (z) =
    \frac{Mc^{- 1} \alpha}{(1 - \alpha z)^m} . \label{eq:df final maj}
  \end{equation}
\end{proposition}

\begin{proof}
  The solution~$\hat{h} (z)$ of the homogeneous part of~\eqref{eq:df maj deq}
  with $\hat{h}_0 = 1$ is given by
  \[ \hat{h} (z) = \frac{1}{1 - c \bar{\eta} \varphi (z)} \exp \int_0^z
     \frac{\hat{a} (w)}{1 - c \bar{\eta} \varphi (w)} \mathd w. \]
  Observe that
  \[ \left( \varphi \minmaj{u} \right)' \ll \hat{u}_0  (\varphi \hat{g})' =
     \hat{u}_0  (\varphi' + \varphi \hat{a})  \hat{g} \ll \hat{u}_0  (\varphi'
     + \varphi \hat{a})  \hat{h}, \]
  so that, in Lemma~\ref{lem:df maj ineq}, we can take $\hat{v} (z) =
  \hat{u}_0  (\varphi' + \varphi \hat{a})  \hat{h}$. The method of variation
  of parameters then leads to the expression
  \[ \hat{\delta} (z) = \hat{h} (z)  \left( \hat{\delta}_0 + c \bar{\eta}
     \int_0^z \frac{\hat{v} (w)}{\hat{h} (w)} \mathd w \right) = \hat{h} (z) 
     \left( \hat{\delta}_0 + c \bar{\eta}  \hat{u}_0  \left( \varphi (z) +
     \int_0^z \varphi (w)  \hat{a} (w) \mathd w \right) \right) . \]
  Using the bounds from Lemma~\ref{lem:ipp}
  \[ \int_0^z \frac{\hat{a} (w)}{1 - c \bar{\eta} \varphi (w)} \mathd w \ll
     \frac{z \hat{a} (z)}{1 - c \bar{\eta} \varphi (z)}, \qquad \int_0^z
     \varphi (w)  \hat{a} (w) \mathd w \ll z \varphi (z)  \hat{a} (z), \]
  we see that $\hat{\delta} (z)$
  is bounded by the right-hand side of~\eqref{eq:df final maj}.
  Note in passing that $z \hat{a} (z)$ could be replaced by $\int_0^z \hat{a}$ at the price
  of a slightly more complicated final bound.
\end{proof}

Step~\ref{step:add-error} of Algorithm~\ref{algo:dfsum} effectively computes
an upper bound on $\abs{\delta (\zeta)}$ using inequality~\eqref{eq:df final maj}. It follows
that the returned interval~$\tmmathbf{s}_N$ contains the exact partial sum
$\sum_{n = 0}^{N - 1} u_n \zeta^n$ corresponding to the input data, as stated
in the specification of the algorithm.
This concludes the proof of correctness of Algorithm~\ref{algo:dfsum}.

\begin{remark}
  Instead of $\varepsilon_n / (\abs{\tilde{u}_{n - 1}} + \cdots + \abs{\tilde{u}_{n
  - s}})$, one could compute a run-time bound
  directly on $\abs{\varepsilon_n / \hat{h}_n}$. Doing so leads to a somewhat
  simpler variant of the above analysis. We chose to present the version given
  here because it is closer to plain floating-point error analysis --- and
  gives us an excuse to illustrate the generalization to recurrences with
  polynomial coefficients of the technique of Section~\ref{sec:toy rel}.
  Another small advantage is that plugging in a sharper first-order majorant
  equation as suggested below requires no other change
  to the algorithm, whereas it may not be obvious how to compute a good lower
  bound on~$\hat{h}_n$.
\end{remark}

It is natural to ask how this algorithm compares to naive interval summation.
We limit ourselves to a short informal discussion.
When the operator~$P$, the series~$u (z)$ and the evaluation point~$\zeta$
are fixed, our bound~$\Delta_N$ on the global error decreases linearly with
$\hat{\delta}_0 + \bar{\eta}$. Suppose that we run the algorithm with a
relative working precision of $t$~bits. Under the reasonable assumptions
that $\tmop{rad} (\tmmathbf{u}_n) = \Omicron (2^{- t})$ for $0 \leqslant n <
r$ and that both $\eta_n$ and $\tmop{rad} (\tmmathbf{s}_n)$
are\footnote{Such a growth for~$\bar{\eta}$ is reasonable since the
coefficients~$b_i$ of the recurrence are polynomials.
Regarding~$\tmmathbf{s}_n$, we can in fact expect to have $\tmop{rad}
(\tmmathbf{t}_n) / \zeta^n = \Omicron (n 2^{- t})$, and, since $u_n \zeta^n$
converges geometrically to zero, \ $\tmop{rad} (\tmmathbf{s}_n) \approx
\tmop{rad} (\tmmathbf{s}_{n - 1}) + \Omicron (nu_n \zeta^n 2^{- t}) +
\Omicron (2^{- t})$, leading to $\tmop{rad} (\tmmathbf{s}_n) = \Omicron (n
2^{- t})$.} $\Omicron (n^d 2^{- t})$ for some~$d$, we then have $\Delta_N =
\Omicron (N^d 2^{- t})$. As the truncation order~$N$ necessary for reaching
an accuracy $\abs{\tmmathbf{s}_N - u (\zeta)} \leqslant 2^{- q}$ is $N =
\Omicron (q)$, this means that, for fixed $u$ and $\zeta$, the algorithm
needs no more than $q + \Omicron (\log q)$ bits of working precision to
compute an enclosure of $u (\zeta)$ of width $2^{- q}$.
In the same setting, computing the partial sum purely in ball arithmetic
(that is, without Step~\ref{step:round} of Algorithm~\ref{algo:dfsum})
may require a working precision of the order of $q + \lambda N$ bits,
for some~$\lambda$ depending on the recurrence.

The majorant series of Proposition~\ref{prop:df maj
series} was chosen to keep the algorithm simple, not to optimize the error
bound, so that we do not expect the version described here to be practical.
One helpful feature it does have is that the parameter $\alpha$ can be taken
arbitrarily close to~$\rho (p_r)^{- 1}$ without forcing other parts of the
bound to tend to infinity. (This is in contrast with the geometric majorant
series typically found in textbook proofs of theorems on differential
equations.) Nevertheless, the exponential factor in~\eqref{eq:df final rad}
can easily grow extremely large, and we expect Algorithm~\ref{algo:dfsum} to
lead to unusable bounds in practice on moderately complicated examples. Even
in simple cases, forcing a majorant series of finite radius of convergence
when $p_r$ is constant is far from optimal.

The same idea, though, can be used with a more sophisticated choice of
majorant series. In particular, the algorithm adapts without difficulty if
$\hat{a} (z)$ is a sharper rational majorant of the coefficients of the
equation.
As a first step toward making the algorithm practical,
we have implemented a variant based
on the more flexible framework of~{\cite{Mezzarobba2019}}, in the
ore\_algebra
package~{\cite{KauersJaroschekJohansson2015,Mezzarobba2016}} for SageMath.
Preliminary experiments suggest that, in some cases, it is very effective in
reducing the working precision necessary for computing enclosures of solutions
of differential equations with polynomial coefficients.
At this stage, though, it does not consistently run faster than naive interval
summation, due both to overestimation issues and to the computational overhead
of obtaining good bounds.
We leave it to future work to develop an efficient Taylor method for solving
linear ODEs with polynomial incorporating the technique presented in this section.
It would also be interesting to extend the analysis to the
computation of ``logarithmic series'' solutions of linear ODEs at regular
singular points.

\section*{Acknowledgments}

This work benefited from remarks by many people, including Alin Bostan,
Richard Brent, Thibault Hilaire, Philippe Langlois, and
Nicolas Louvet. The initial impulse came from discussions with Fredrik
Johansson, Guillaume Melquiond, and Paul Zimmermann. Frédéric Chyzak
suggested to work with elements of $\Omega [[t]]$ instead of $\mathbb{R}
[x^{\pm 1}] [[t]]$ in Section~\ref{sec:wave}. I am especially grateful to
Guillaume Melquiond, Anne Vaugon, and two anonymous referees for many
insightful comments,
and to Paul Zimmermann for his thorough reading of a preliminary
version.

\printbibliography

\end{document}